\documentclass[12pt]{amsart}

\usepackage{euscript}
\usepackage{amsthm}
\usepackage{amssymb,amscd}
\usepackage{amsmath}
\usepackage{graphicx}
\usepackage[usenames,dvipsnames]{xcolor}
\usepackage[colorlinks]{hyperref}
\usepackage{enumerate}
\usepackage{ulem}
\let\ophi\phi
\let\phi\varphi
\newcommand{\privateE}[1]{#1^0}
\newcommand{\privateOm}[1]{#1'}
\newcommand{\PE}{\privateE{P}}
\newcommand{\POm}{\privateOm{P}}
\newcommand{\POmb}{\privateOm{\bar P}}
\newcommand{\phib}{\bar\phi}
\newcommand{\phif}{\Phi}

\newcommand{\LE}{\privateE{L}}
\newcommand{\LOm}{\privateOm{L}}
\newcommand{\lE}{\privateE{l}}
\newcommand{\lOm}{\privateOm{l}}
\newcommand{\pEm}{p}

\newtheorem*{remark*}{Remark}
\newtheorem{theorem}{Theorem}
\newtheorem*{otheorem}{Theorem}
\newtheorem*{mthm}{Main Theorem}
\newtheorem*{birkhoffConj}{Birkhoff Conjecture}
\newtheorem{proposition}[theorem]{Proposition}

\newtheorem{corollary}[theorem]{Corollary}
\newtheorem{lemma}[theorem]{Lemma}
\newtheorem{remark}[theorem]{Remark}

\newtheorem*{definition}{Definition}

\makeatletter

\newcommand\mergedsub[2]{#1\sc@sub{#2}}
\newcommand\newsubcommand[3]{\newcommand#1{\mergedsub{#2}{#3}}}
\def\sc@sub#1{\def\sc@thesub{#1}\@ifnextchar_{\sc@mergesubs}{_{\sc@thesub}}}
\def\sc@mergesubs_#1{_{\sc@thesub,#1}}
\makeatother

\newcommand{\inv}{^{-1}}
\newcommand{\crit}{\text{c}}

\newcommand{\matr}[4]{\left(\begin{array}{cc}#1&#2\\#3&#4\end{array}\right)}
\newcommand{\depp}[1]{[#1]}
\newcommand{\TOm}{\T_\Om}
\newcommand{\intr}{\textup{int}\,}

\newcommand{\cb}{\Xi}
\newcommand{\conq}{w_q}
\newcommand{\hth}{\textup{h}}
\newcommand{\trl}{\textup{t}}

\newcommand{\ecc}{\textup{hr}}
\newcommand{\fL}{f_{\text{L}}} 
\newcommand{\ksp}{{k_*}}
\newcommand{\Z}{\mathbb{Z}}
\newcommand{\R}{\mathbb{R}}
\newcommand{\Q}{\mathbb{Q}}

\newcommand{\T}{\mathbb{T}}

\newcommand{\cL}{\mathcal{L}}
\newcommand{\lazu}{\Psi_{\text{L}}}
\newcommand{\cD}{\mathcal{D}}
\newcommand{\Om}{\Omega}
\newcommand{\om}{\omega}
\newcommand{\tet}{\theta}
\newcommand{\Id}{\textup{Id}}
\newcommand{\bfn}{{\bf n}}

\newcommand{\eF}{e^{\mathbb F}}

\newcommand{\Spec}{\text{Spec}\,}
\newcommand{\hDist}{\text{dist}_{\text{H}}}
\newcommand{\hGm}{\widehat\Gm}
\newsubcommand{\GmL}{\Gm}{\text{L}}
\newcommand{\kPar}{A}
\newcommand{\kEll}{\mathbb{E}}
\newcommand{\bbar}[1]{\bar{\bar{#1}}}

\newcommand{\rsmoothness}{{39}}
\newcommand{\rcloseness}{{1}}
\newcommand{\rsuperexp}{{703/702}}
\newcommand{\basis}{\mathcal B}
\def\om{\omega}
\def\al{\alpha}
\def\eps{\varepsilon}

\def\R{\mathbb R}
\def\T{\mathbb T}
\def\Q{\mathbb Q}
\def\Z{\mathbb Z}

\def\cE{\mathcal E}

\def\~{\tilde}

\def\Gm{\Gamma}

\def\bt{\beta}
\def\dt{\delta}
\def\lb{\lambda}

\newcommand{\nl}[1]{_{L^{#1}}}
\newcommand{\nlmu}[1]{_{L_\mu^{#1}}}
\newcommand{\nol}[1]{_{L^{#1}\to L^{#1}}}
\newcommand{\yq}{\eta_q}

\def\bdef{\begin{definition}}
\def\endef{\end{definition}}
\def\bthm{\begin{theorem}}
\def\ethm{\end{theorem}}
\def\blm{\begin{lemma}}
\def\elm{\end{lemma}}
\def\brm{\begin{remark*}}\def\erm{\end{remark*}}%
\def\bprop{\begin{proposition}}
\def\eprop{\end{proposition}}
\def\bcor{\begin{corollary}}
\def\ecor{\end{corollary}}
\def\be{\begin{eqnarray}}
\def\ee{\end{eqnarray}}
\def\beal{\begin{aligned}}
\def\enal{\end{aligned}}

\begin{document}
\title[Integrable deformations of ellipses of small
eccentricity]{An integrable deformation of an ellipse of
  small eccentricity is an ellipse}

\author{Artur Avila}
\address{Artur Avila\\
  CNRS, IMJ-PRG, UMR 7586, Univ Paris Diderot, Sorbonne
  Paris Cit\'e, Sorbonnes Universit\'es, UPMC Univ Paris 06,
  F-75013, Paris, France \& IMPA, Estrada Dona Castorina
  110, Rio de Janeiro, Brasil }
\author{Jacopo De Simoi}
\address{Jacopo De Simoi\\
  IMJ-PRG Universit\'e Paris Diderot, 75013 Paris, France \&
  Department of Mathematics\\
  University of Toronto\\
  40 St George St. Toronto, ON, Canada M5S 2E4} \email{{\tt
    jacopods@math.utoronto.ca}}
\urladdr{\href{http://www.math.utoronto.ca/jacopods}{http://www.math.utoronto.ca/jacopods}}
\author{Vadim Kaloshin}
\address{Vadim Kaloshin\\
  Department of Mathematics\\
  University of Maryland, College Park\\
  20742 College Park, MD, USA.}  \email{{\tt
    vadim.kaloshin@gmail.com}}
\urladdr{\href{http://www2.math.umd.edu/~vkaloshi/}{http://www2.math.umd.edu/~vkaloshi/}}

\maketitle
\begin{abstract}
  The classical Birkhoff conjecture claims that the boundary
  of a strictly convex integrable billiard table is
  necessarily an ellipse (or a circle as a special case).
  In the paper we show that a version of this conjecture is
  true for tables bounded by small perturbations of ellipses
  of small eccentricity.
\end{abstract}

\section{Introduction}

Let $\Om \subset \R^2$ be a strictly convex domain; we say
that $\Om$ is $C^r$ if its boundary is a $C^r$-smooth curve.
We consider the billiard problem inside $\Om$, which is then
commonly called the ``billiard table''.  The problem was
first investigated by Birkhoff (see~\cite{Bf}) and is
described as follows: a massless billiard ball moves with
unit speed and no friction following a rectilinear path
inside the domain $\Om$.  When the ball hits the boundary,
it is reflected elastically according to the law of optical
reflection: the angle of reflection equals the angle of
incidence.  Such trajectories are called \emph{broken
  geodesics}, {as they correspond to local minimizers of the
  distance functional.}

We call a (possibly not connected) curve $\hGm \subset \Om$
a \emph{caustic} if any billiard orbit having \emph{one} segment
tangent to $\hGm$ has \emph{all} its segments tangent to $\hGm$.

We call a billiard $\Om$ \emph{locally integrable} if the
union of all caustics has nonempty interior; likewise, a
billiard $\Om$ is said to be \emph{integrable}
(see~\cite{Gu}) if the union of all \emph{smooth convex}
caustics, denoted $\mathcal C_\Om$, has nonempty interior.

It follows by rather elementary geometrical considerations,
(but see e.g.~\cite[Theorem 4.4]{Taba} for a detailed proof)
that a billiard in an ellipse is integrable: its caustics
are indeed co-focal ellipses and hyperbolas.  {A long
standing open question asks whether or not there exist
integrable billiards which are different from ellipses. }
\begin{birkhoffConj}[see\footnote{ The conjecture,
    classically attributed to Birkhoff, can be found in
    print only in~\cite{Po} by H. Poritsky, who worked with
    Birkhoff as a post-doctoral fellow in the years
    1927--1929.}~\cite{Po}, \cite{Gu}]
  If the billiard in $\Om$ is integrable, then $\partial\Om$
  is an ellipse.
\end{birkhoffConj}
The most notable result related to the Birkhoff Conjecture
is due to Bialy (see~\cite{Bi} but also~\cite{Woj}) who
proved that, if convex caustics completely foliate $\Om$,
then $\Om$ is necessarily a disk.  On the other hand, it is
simple to construct smooth (but not analytic) locally
integrable billiards different from ellipses.  In fact, it
suffices to perturb an ellipse away from a neighborhood of
the two endpoints of the minor axis.  More interestingly,
Treschev (see~\cite{Tre}) gives indication that there are
analytic locally integrable billiards such that the dynamics
around one elliptic point is conjugate to a rigid rotation.

There is a remarkable relation between properties of the
billiard dynamics in $\Om$ and the spectrum of the Laplace
operator in $\Om$.  Given a smooth domain $\Om$, the length
spectrum of $\Om$ is defined as the collection of perimeters
of its periodic trajectories, counted with multiplicity:
\[
\mathcal L_\Om:=\mathbb{N}\{\text{lengths of periodic
  trajectories in }\Om\} \cup \mathbb{N}\ell_{\partial\Om},
\]
where $\ell_{\partial\Om}$ denotes the length of $\partial\Om$.

Let $\Spec\Delta$ denote the spectrum of the Laplace
operator in $\Om$ with (e.g) Dirichlet boundary
condition\footnote{ From the physical point of view, the
  Dirichlet eigenvalues $\lb$ correspond to the
  eigenfrequencies of a membrane of shape $\Om$ which is
  fixed along its boundary.  }, i.e. the set of $\lb$ so
that
\begin{align*}
  \Delta u &= \lb u, &
  u &= 0 \text{ on }\partial \Om.
\end{align*}

Andersson--Melrose (see~\cite[Theorem (0.5)]{AM}, which
substantially generalizes some earlier result
by~\cite{Chaza,Dui}) proved that, for strictly convex
$C^{\infty}$ domains, the following relation between the
\emph{wave trace} and the length spectrum holds:
\begin{align*}
  \textup{sing\ supp}\left(t\mapsto \sum_{\lb_j\in \Spec \Delta} \exp (i
  \sqrt{-\lb_j}t)\right)\subset\pm\mathcal
  L_\Om\cup\{0\}.
\end{align*}
Generically (i.e.\ when each element of the length spectrum
has multiplicity one and the corresponding periodic orbits
satisfy a non-degeneracy condition) the above inclusion
becomes an equality and the Laplace spectrum determines the
length spectrum (see e.g.~\cite{Popov} and references
therein). 

This is, of course, related to inverse spectral theory and
to the famous question by M. Kac~\cite{Ka}: ``Can one hear
the shape of a drum?'', which more formally translates to
``Does the Laplace spectrum determine a domain?''  There is
a number of counterexamples to this question (see
e.g. \cite{GWW,Sunada,Vi}), but the domains considered in
such examples are neither smooth nor convex.

{In~\cite{Sa}, P. Sarnak conjectures that the set of
  smooth convex domains isospectral to a given smooth convex
  domain is finite.}  Hezari--Zelditch, {going in the
  affirmative direction,} proved in~\cite{HZ} that, given an
ellipse $\cE$, any one-parameter $C^\infty$-deformation
$\Om_\eps$ which preserves the Laplace spectrum (with
respect to either Dirichlet or Neumann boundary conditions)
and the $\mathbb Z_2\times\mathbb Z_2$ symmetry group of the
ellipse has to be \emph{flat} (i.e., all derivatives have to
vanish for $\eps = 0$).  Further historical remarks on the
inverse spectral problem can also be found in~\cite{HZ}.
\section{Our main result}
\label{sec:weak-main}
Given a strictly convex domain $\Om$, we define the
associated billiard map $f_{\Om}$ as follows.  Let us fix a
point $P_0\in\partial\Om$ and denote with $s$ the arc-length
parametrization of $\partial\Om$ starting at $P_0$ in the
counter-clockwise direction; let $P_s$ denote the point on
$\partial \Om$ parametrized by $s$.  We define the billiard
map
\begin{align}
  \label{eq:billiard-map}
  f_\Om: \TOm\times [0,\pi]&\to \TOm\times [0,\pi], \\
  (s,\varphi) &\mapsto (s',\varphi'),\notag
\end{align}
where $\TOm = \R/\ell_{\partial\Om}\Z$, $\ell_{\partial\Om}$
is the length of $\partial\Om$, {$P_{s'}$} is the reflection
point of a ray leaving {$P_s$} with angle $\varphi$ with
respect to the counter-clockwise tangent ray to the boundary
$\partial\Om$ and $\varphi'$ is the angle of incidence of
the ray at $P_{s'}$ with the clockwise tangent.  If there is
no confusion we will drop the subscript $\Om$ and simply
refer to the billiard map as $f$ and let $\T = \T_\Om$.

In the remaining part of this paper, we agree that all
caustics that we will consider will be smooth and convex;
we will refer to such curves simply as \emph{caustics}.

Let $\hGm$ be a caustic for $\Om$; for any $s\in\TOm$ there
exist two rays leaving $P_s$ which are tangent to $\hGm$,
one aligned with the counter-clockwise tangent of $\hGm$ and
the other one with the clockwise tangent; let us denote with
$\phi^{\pm}_{\hGm}(s)$ their corresponding angles of
reflection.  Observe that, by reversibility of the dynamics,
the trajectory associated with $\phi^-$ is the time-reversal
of the trajectory associated with $\phi^+$, i.e.
$\phi^- = \pi-\phi^+$.  We can, thus, restrict our analysis
to (e.g.)\ $\phi^+$; in doing so we will drop, for
simplicity, the superscript $+$ from our notations.

The graph $\Gm = \{(s,\phi_{\hGm}(s))\}_{s\in\T}$ \, is, by
definition of a caustic, a (non-contractible) $f$-invariant
curve\footnote{ Indeed, by Birkhoff's Theorem, any
  $f$-invariant non-contractible curve is a Lipschitz
  graph.}. Therefore, the restriction $f|_{\Gm}$ is a
homeomorphism of the circle, and, as such, it admits a
rotation number, which we denote with $\om$.  In fact (since
we have chosen $\phi^+$ over $\phi^-$), we always have
$0 < \om\le1/2$.
\begin{definition}
  We say $\hGm$ is an \emph{integrable rational caustic} if
  the corresponding (non-contractible) invariant curve $\Gm$
  consists of periodic points; in particular, the
  corresponding rotation number is rational.  If $\Om$
  admits integrable rational caustics of rotation number
  $1/q$ for all $q>2$, we say that $\Om$ is \emph{rationally
    integrable}.
\end{definition}
\begin{remark*}
  A more standard definition of integrability requires
  existence of a ``nice'' first integral. Existence of a
  ``nice'' first integral for a billiard does not imply
  integrability of any caustic of rational rotation number.
  For instance, the invariant curve corresponding to points
  belonging to the coinciding separatrix arcs of a
  hyperbolic periodic orbit of $f$ is not integrable.
\end{remark*}
The following lemma provides a sufficient (although a priori
weaker) condition for rational integrability.
\begin{lemma}
  Assume the interior of the union of all smooth convex
  caustics $\intr{}\mathcal C_\Om$ of a billiard $\Om$
  contains caustics of rotation number $1/q$ for any
  $q\ge2$, then $\Om$ is rationally integrable.
\end{lemma}

\begin{proof}
  It is known that if a caustic with rational rotation
  number belongs to the interior of a foliation with
  caustics, then it is integrable (see e.g.~\cite[Corollary
  4.5]{Taba} for the general statement and~\cite[Proposition
  2.8]{GM1} for the special case of an ellipse).  Thus, our
  assumption guarantees the rational integrability of $\Om$.
\end{proof}
Let us denote with $\cE_e \subset \R^2$ an ellipse of
eccentricity $e$ and perimeter {$1$}.
\begin{mthm}
  There exists $e_0>0$ and $\eps_0 > 0$ such that for any
  $0\le e\le e_0$, $0\le \eps < \eps_0$, any rationally
  integrable $C^\rsmoothness$-smooth domain $\Om$ so that
  $\partial\Om$ is $C^{\rsmoothness}$-$\eps$-close to
  $\cE_e$ is an ellipse.
\end{mthm}
\begin{remark*}
  We will indeed prove a slightly stronger version of the
  above theorem, stated as Theorem~\ref{t_main}.
\end{remark*}
\begin{remark*}
  Our requirements for smoothness are probably not optimal,
  but they are crucial for the approach used in our proof
  (see the proof of Lemma~\ref{l_inductive} and, in
  particular, Footnote~\ref{fn_rsmoothness}).  One could
  possibly relax them using~\cite{Bu}.
\end{remark*}
\textsc{Acknowledgments:} We thank L. Bunimovich, D.
Jakobson, I. Polterovich, A. Sorrentino, D. Treschev, J. Xia, S. Zelditch
and the anonymous referee for their most useful comments which
allowed to vastly improve the exposition of our result. JDS
acknowledges partial NSERC support. VK acknowledges
partial support of the NSF grant DMS-1402164.

\section{Our strategy and the outline of the
  paper}\label{s_strategy}
Let us start by exploring the simplified setting of
integrable infinitesimal deformations of a circle; we then
use this insight to describe the main strategy of our proof
in the general case.  Let $\Om_0$ be the unit disk and let
us denote polar coordinates on the plane with $(r,\ophi)$.
Let $\Om_\eps$ be a one-parameter family of deformations
given in polar coordinates by
$\partial\Om_\eps=\{(r,\ophi)=(1+\eps
n(\ophi)+O(\eps^2),\ophi)\}$.
Consider the Fourier expansion of $n$:
\begin{align*}
  n(\ophi)=n_0 + \sum_{k > 0} n'_k\sin (k\ophi)+ n''_k \cos (k\ophi).
\end{align*}
\begin{otheorem}[Ramirez-Ros~\cite{RR}] If $\Om_\eps$ has an
  integrable rational caustic $\Gm_{1/q}$ of rotation number
  $1/q$ for all sufficiently small $\eps$, then $n'_{kq} =
  n_{kq}'' = 0$ for any $k\in\mathbb N$.
\end{otheorem}
Let us now assume that the domains $\Om_\eps$ are rationally
integrable for all sufficiently small $\eps$: then the above
theorem implies that $n'_k = n''_k = 0$ for $k > 2$, i.e.
\begin{align*}
  n(\ophi)&=n_0+n'_1\cos \ophi+n_1''\sin \ophi+
            n'_2\cos 2\ophi+n_2''\sin 2\ophi\\
          &= n_0+n_1^*\cos (\ophi-\ophi_1)+
            n_2^*\cos 2(\ophi-\ophi_2)
\end{align*}
where $\ophi_1$ and $\ophi_2$ are appropriately chosen
phases.
\begin{remark}\label{r_motionDescription}
  Observe that\nopagebreak
  \begin{itemize}
  \item $n_0$ corresponds to an homothety;
  \item $n_1^*$ corresponds to a translation in the
    direction forming an angle $\ophi_1$ with the polar axis
    ($\{\ophi = 0\}$);
  \item $n_2^*$ corresponds to a deformation into an ellipse
    of small eccentricity with the major axis meeting the
    polar axis at the angle $\ophi_2$.
  \end{itemize}
  This implies that, \emph{infinitesimally} (as
  $\eps\to 0$), rationally integrable deformations of a
  circle are tangent to the $5$-parameter family of
  ellipses.
\end{remark}
Observe that in principle, in the above theorem, one may
need to take $\eps\to0$ as $q\to\infty$.  On the other hand,
we are studying a situation in which $\eps > 0$ is small but
not infinitesimal; hence we cannot use directly the above
theorem to prove our result, and we need to pursue a more
elaborate strategy, which we now describe.

Let $\Om_0$ be a strictly convex domain {(to fix ideas
  the reader may assume $\Om_0$ to be an ellipse)} and
consider a tubular neighborhood $U_{\Om_0}$ of
$\partial \Om_0$ so that for any $P\in U_{\Om_0}$ we can
associate the \emph{tubular coordinates} $(s,n)$, where $s$
is the $s$-coordinate of the orthogonal projection of $P$
onto the boundary $\partial \Om_0$ and ${n}$ is the oriented
distance of $P$ along the orthogonal direction to
$\partial\Om_0$ defined so that $n>0$ outside (resp.\
$n < 0$ inside) of $\Om_0$.

We can, thus, identify any given domain $\Om$ so that
$\partial\Om\subset U_{\Om_0}$ with the graph of a function
$\bfn(s)$ in tubular coordinates.  In order to do that one
can \emph{project} points from $\partial \Om$ to
$\partial \Om_0$ and \emph{lift} points from
$\partial \Om_0$ to $\partial \Om$.  In the sequel we will
only consider perturbations $\Om$ which can be described by
a function $\bfn(s)$ of this form and we introduce the
following (slightly abusing, but suggestive) notation
\begin{align*}
  \partial\Om=\partial\Om_0 + \bfn.
\end{align*}
Our strategy now proceeds as follows: let $\Om_0$ be an
ellipse {$\mathcal E_e$ of eccentricity $e$ and
  perimeter $1$}; in particular,
all rational caustics of rotation number $1/q$ for $q > 2$ are
integrable.\\%

\noindent\textbf{Step 1:} We derive a quantitative necessary
condition for preservation of an integrable rational caustic
(see Theorem~\ref{perturbation-perimeter} in
Section~\ref{sec:aa-necessary-cond}).\\%

\noindent\textbf{Step 2:} We define Deformed
Fourier modes for the case of ellipses; they will be denoted
by $\{c_0,c_q,s_q:q > 0\}$ and satisfy the following
properties:
\begin{itemize}
\item {\it (relation with Fourier Modes)} There exists (see
  Lemma~\ref{error-estimate-2}) $C^*(e)>0$ with $C^*(e)\to0$
  as $e\to0$ so that $\|c_0-1\|_{C^0}\le C^*(e)$ and for any $q\ge 1$
  \begin{align*}
 \|c_q-\cos (2\pi q\cdot)\|_{C^0} &\le C^*(e)/q,
&\|s_q-\sin (2\pi q\cdot)\|_{C^0} &\le C^*(e)/q.
  \end{align*}
\item (\emph{transformations preserving integrability}) We
  define (in Section~\ref{selected-five}) the functions
\begin{align*}
  c_0,\,c_1,\,s_1,\,c_2,\,s_2
\end{align*}
having the same meaning described in the previous
remark: they generate homotheties, translations and
hyperbolic rotations about an arbitrary axis.

\item ({\it annihilation of inner products})
Let $\bfn$ identify a $C^r$ deformation of $\Om_0$ and consider,
for $\eps\in(-\eps_0,\eps_0)$, the one-\hskip0pt{}parameter family of
domains
\begin{align*}
  \partial\Om_\eps:=\partial\Om_0+\eps\bfn.
\end{align*}
For any $q > 2$, we define (in
Section~\ref{s_DeformedFourierModes}) functions $c_q,s_q$ so
that if $\Om_\eps$ has an integrable rational caustic
$\hGm^\eps_{1/q}$ of rotation number $1/q$ for all
sufficiently small $\eps$, then
\begin{align}
  \label{eq:vanishing}
  \langle\bfn,c_q\rangle =  \langle\bfn,s_q\rangle = 0,
\end{align}
where $\langle\cdot,\cdot\rangle$ is a weighted $L^2$ inner
product.  In fact, in Lemma~\ref{l_projection-estimate} we
derive a perturbative version of the above {infinitesimal
  orthogonality} conditions.  More precisely: if, for some
sufficiently $C^1$-small, $C^5$-perturbation $\bfn$, the
domain bounded by $\partial \Om=\partial \Om_0+\bfn$ has an
integrable rational caustic $\hGm_{1/q}$, then we can
replace~\eqref{eq:vanishing} with
\begin{align}\label{eq:vanishing-nonuniform}
  \langle\bfn,c_q\rangle &=O(q^8\|\bfn\|_{C^1}^2) , &\langle\bfn,s_q\rangle&=O(q^8\|\bfn\|_{C^1}^2).
\end{align}
Observe that, as we hinted at earlier, the above estimate is
necessarily \emph{non-uniform in $q$}.
Notice that the functions $c_q,s_q$ can be explicitly
defined using elliptic integrals via action-angle
coordinates (see~\eqref{dynam-basis}).
\item (\emph{linear independence}) For sufficiently small
  eccentricity (see Section~\ref{main-nonfamily}), the
  functions $\{c_0,c_q, s_q\,:\,q>0\}$ form a (non-orthogonal) basis of
  $L^2$.
\end{itemize}

\

\noindent\textbf{Step 3:} We then conclude the proof (in Section~\ref{s_theProof})
using the following approximation result
(Lemma~\ref{l_inductive}): if $\Om_\eps$ is rationally
integrable and $\partial\Om_\eps$ is
an $O(\eps)$-perturbation of an ellipse $\partial\Om_0=\cE_e$
of small eccentricity $e$, then there exists an ellipse
$\bar\cE$ such that $\partial\Om_\eps$ is an
$O(\eps^{\beta})$-perturbation of $\bar\cE$ for some
$\beta > 1$. This step is done as follows
\begin{itemize}
\item For a fixed $\eps=\|\bfn\|_{C^1}$ and each
  $2 < q \le q_0(\eps)=[\eps^{-1/9}]$,
  condition~\eqref{eq:vanishing-nonuniform} implies that the
  size of the $q$-th generalized Fourier coefficients is
  small and, therefore, their sum up to $q_0$ is bounded by
  $\eps^\bt$.
\item Due to decay of the generalized Fourier coefficients
  we can also show that the sum over $q>q_0$ is bounded by
  $\eps^\bt$.
\end{itemize}
Combining the above estimates, we gather that
$\partial \Omega_\eps$ can be approximated by an ellipse
$\bar {\cE}$ with an error $O(\eps^\bt)$, where $\bar\cE$ is
the ellipse generated by projecting $\bfn$ onto the subspace
generated by the first $5$ Deformed Fourier modes.  Applying
this result to the best approximation of
$\partial\Omega_\eps$ by an ellipse, we obtain a
contradiction unless $\partial\Omega_\eps$ is itself an
ellipse.
\begin{remark*}
  We emphasize that our condition on eccentricity is not an
  abstract smallness assumption.  More specifically: one has
  to check that some explicit condition on the eccentricity
  (given in~\eqref{eccentr-smallness-assumption}) holds
  true.
\end{remark*}
\section{A sufficient condition for rational integrability,
  the Deformation Function, and action-angle variables}
\label{sec:aa-necessary-cond}
Let $\Om_0 = \cE_e\subset\R^2$ be an ellipse of eccentricity
$e$ and perimeter $1$; let $f=f_{\cE_e}$ be the associated
billiard map.  For convenience, let us fix $P_0$ be one of
the end-points of the major axis.  For $0<\om<1/2$, let
$\hat \Gm_\om$ be the caustic of rotation number $\om$ and
$\Gm_\om$ be the corresponding invariant curve of $f$.
Then, for any $\om$, there exists a parametrization $\theta$
of $\cE_e$ so that $f$ acts as a rigid rotation of angle
$\om$, i.e.\ if $S(\theta;\om)$ denotes the change of
variables from the $\tet$-parametrization to the arc-length
parametrization, for any $\tet\in\T$ we have:
\begin{align}
  f(S(\tet;\om),\phif(\tet;\om))=
(S(\tet+\om;\om),\phif(\tet+\om;\om)),
\label{e_dynamicalS}
\end{align}
where we introduced the shorthand notation
$\phif(\tet;\om) = \phi_{\hGm_\om}(S(\tet;\om))$.  In other
words, $(S,\phif)$ is the change of variables from the
\emph{action-angle coordinates} $(\theta,\om)$ to arc-length
and reflection angle.  Geometrically: given $S(\tet;\om)$,
consider the trajectory leaving $P_{S(\tet;\om)}$ with angle
$\phif(\tet;\om)$; this ray will be tangent to $\hGm_\om$
and land at the point parametrized by $S(\tet+\om;\om)$ with
angle $\phif(\tet+\om;\om)$ with respect to the tangent to
$\cE_e$ at $S(\tet+\om;\om)$.

We normalize $S$ so that $S(0;\om)=0$ for all
$\om\in (0,1/2)$.  Following Tabanov (see~\cite{Ta}) we can
assume $S$ and $\phif$ to be analytic in both $\tet$ and
$\om$.  In particular, for each $\om \in (0,1/2)$ the map
$S(\cdot;\om)$ is an (analytic) circle diffeomorphism.
Observe additionally that both functions depend analytically
on the parameter $e$ and, moreover, for $e = 0$ we have
$S(\tet;\om) = \tet$ and $\phif(\tet;\om) = \pi\om$.

Let now $\Om$ be a deformation of $\cE_e$ identified by a
$C^{\rsmoothness}$ function $\bfn$.  Given
$p/q\in\Q\cap(0,1/2)$ with $p$ and $q$ relatively prime, let
us define the \emph{Deformation Function}:
\begin{align} \label{deformed-function}
  \cD\left(\bfn,S,\phif,\frac pq\right)(\tet)= 2\sum_{k=1}^q
  \bfn \left(S\left(\tet+ k\frac p q; \frac pq\right)\right)
  \sin \phif\left(\tet+k\frac pq; \frac pq\right).
\end{align}
In Theorem~\ref{perturbation-perimeter} below we show that
the Deformation Function is the leading term of the change
of perimeter of the possibly non-convex polygon inscribed in
$\cE_e$ corresponding to an orbit of rotation number $p/q$
starting at $P_{S(\tet)}$.  In order state more precisely
the above consideration, we now proceed to introduce some
further notation.

First, since in the present article we are interested only
in caustics of rotation number $1/q$, we restrict the
analysis to this case.  Let us thus introduce the convenient
shorthand notations $S_q = S(\cdot, 1/q)$ and
$\phif_q = \phif(\cdot,1/q)$.  Recall that for any ellipse
$\cE_e$, every caustic $\hGm_{1/q}$ of rotation number $1/q$
with $q > 2$ is an integrable rational caustic.  Recall also
that, for any $0\le s<1$, $P_s$ denotes the point whose
arc-length distance from $P_0$ in the counter-clockwise
direction equals $s$.  Define
\begin{align*}
\PE_k(\tet)&=P_{S_q(\tet+k/q)} &\text{for }k&=0,\cdots,q-1.
\end{align*}
In other words, for any $\tet\in \T$ we associate the
corresponding $q$-periodic orbit tangent to the caustic
$\hGm_{1/q}$ given by the points
$\PE_0(\tet),\cdots,\PE_{q-1}(\tet)$.  The variational
characterization of periodic orbits (see e.g.~\cite{Bf})
implies that periodic orbits are given by the vertices of an
inscribed convex $q$-gon with one vertex at $P_{S_q(\tet)}$
and whose perimeter is a stationary value.  Let
$\LE_q(\tet)$ be the perimeter of this $q$-gon, i.e.
\begin{align*}
  \LE_{q}(\tet)=\sum_{k=0}^{q-1} \|\PE_{k+1}(\tet)-\PE_{k}(\tet)\|,
\end{align*}
where $\|\cdot\|$ is the Euclidean distance. Then, since
$\hGm_{1/q}$ is an integrable rational caustic, we conclude
that $\LE_{q}(\tet)$ is actually constant in $\tet$.  
In fact, all periodic orbits belonging to a smooth
one-parameter family have the same, constant, perimeter.

Let us denote with $\POm_0(\tet)\in\partial\Om$ the lift of
$\PE_0(\tet)\in\partial\Om_0$ to $\partial \Om$. Since $\Om$ is strictly
convex, for each $\tet\in \T$, there is a convex $q$-gon
starting at $\POm_0(\tet)$ of maximal perimeter.  Denote
its vertices by $\POm_k(\tet),\ k=0,\cdots,q-1$ and its
perimeter by
\begin{align*}
  \LOm_{q}(\tet)=\sum_{k=0}^{q-1} \|\POm_{k+1}(\tet)-\POm_{k}(\tet)\|.
\end{align*}
If, moreover, $\Om$ admits an integrable rational caustic of
rotation number $1/q$, then the points
$\POm_0(\tet),\cdots,\POm_{q-1}(\tet)$ are actually the
reflection points of the $q$-periodic orbit of rotation
number $1/q$ starting at $\POm_{0}(\tet)$.  By the arguments
given above, $\LOm_q(\tet)$ is also constant.
\begin{theorem}\label{perturbation-perimeter}
  Let $\Om_0 = \cE_e$ be an ellipse of eccentricity
  $0\le e<1$ and perimeter $1$, and let $(S,\phif)$ be the
  corresponding functions defined above.  Then there is
  $c=c(e)>0$ such that for any integer $q$, $q>2$ and $C^5$
  deformation $\partial \Om:=\cE_e+\bfn$ so that $\Om$
  admits an integrable rational caustic $\Gamma_{1/q}$ of
  rotation number $1/q$ and $\ q^8\|\bfn\|_{C^1}<c$:
  \begin{align*}
    \max_\tet \left| L_{q}'(\tet)-L_{q}^0(\tet)-
    \cD(\bfn,S,\phif;1/q)(\tet)
    \right| \le C\,q^8  \|\bfn\|_{C^1}^2,
  \end{align*}
  where $C = C(e,\|\bfn\|_{C^5})$ depends on the
  eccentricity $e$ and monotonically on the $C^5$-norm of
  $\bfn$, but is independent of $q$.
\end{theorem}
\begin{remark*}
  Notice that in~\cite[Proposition 11]{PR} a different
  (weaker, but cleaner) version of this statement is given,
  where it suffices to know only $S(\tet,\om)$.  We also
  point out that $c(e)\to 0$ as $e\to 1$.
\end{remark*}
\begin{proof}[Proof of Theorem~\ref{perturbation-perimeter}]
  Let $\al_k(\tet)$ be the angle between
  $\POm_k(\tet) -\PE_k(\tet)$ and the positive tangent to
  $\cE_e$ at $\PE_k(\tet)$ (see
  Figure~\ref{fig:perturbed-orbits}).  We assume
  $\al_k(\tet)$ to be positive towards the exterior of
  $\cE_e$, i.e. if $\POm_k(\tet)$ is outside of $\cE_e$, then
  $\al_k(\tet) \in (0,\pi)$.  Introduce the displacements
  \begin{align*}
    v_k(\tet)= \|\POm_k(\tet) -\PE_k(\tet)\|
  \end{align*}
  and let $\phi_k(\tet) = \phif_q(\tet+k/q)$.  By definition
  of action-angle coordinates, the edge
  $ \PE_{k+1}(\tet) -\PE_k(\tet)$ has reflection angle
  $\varphi_k(\tet)$ at $\PE_k(\tet)$ and
  $\varphi_{k+1}(\tet)$ at $\PE_{k+1}(\tet)$ respectively.
  Finally, let us introduce the notation
  $\lE_k(\tet) = \|\PE_{k+1}(\tet)- \PE_k(\tet)\|$ and
  $\lOm_k(\tet) = \|\POm_{k+1}(\tet) - \POm_k(\tet)\|$.
  Observe that by Corollary~\ref{edge-uniform}, for each
  $k=0,\cdots,q-1$ we have
  \begin{align}\label{e_boundlE}
    \frac{1}{\cb q} \le \lOm_{k}(\tet)\le
    \frac{\cb}{q}\ \text{ for some }\
    \cb=\cb(e,\|\bfn\|_{C^5}) > 1,
  \end{align}
  and $\cb$ depends monotonically on $\|\bfn\|_{C^5}$.  For
  $k=0,\cdots,q-1$, project $\POm_k(\tet)$ onto $\cE_e$ by the
  orthogonal projection and denote the projected point by
  $\POmb_k(\tet)$.  Observe that, by construction,
  $\POmb_0(\tet) = \PE_0(\tet)$.  Denote, moreover, with
  $\phib^+_k$ (resp. $\phib^-_k$) the angle between
  $\POmb_{k+1}(\tet)-\POmb_k(\tet)$
  (resp. $\POmb_{k}(\tet)-\POmb_{k-1}(\tet)$) and the positive
  (resp. negative) tangent to $\cE_e$ at $\POmb_k(\tet)$ (see
  Figure~\ref{fig:reflect}).

  \begin{lemma}\label{lem:deviated-angle}
    Let $\cb$ be the constant appearing in~\eqref{e_boundlE};
    for any $k=0,\cdots,q-1$: %
    \begin{align*}
      |\phib_k^+-\phib_k^-|\le
      5\cb \,q\,\|\bfn\|_{C^1}.
    \end{align*}
  \end{lemma}
 \begin{figure}[h]
      \begin{center}
        \includegraphics{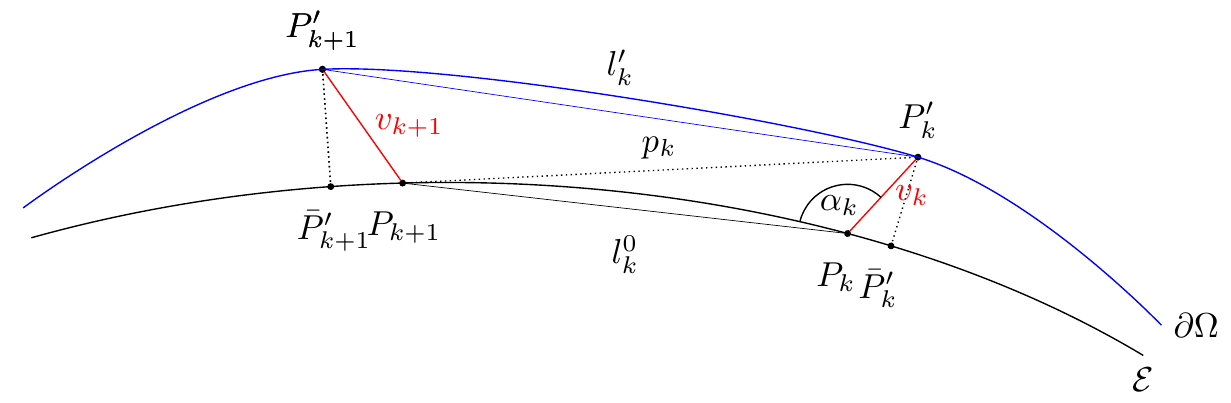}
      \end{center}
      \caption{Two orbits: unperturbed (in black) and
        perturbed (in blue)}
      \label{fig:perturbed-orbits}
    \end{figure}
  \begin{proof}
    Since $\|\POm_k-\POmb_k\|\le \|\bfn\|_{C^0}$ for any
    $k = 0,\cdots,q-1$, the angle between the $k$-th perturbed
    edge and the $k$-th projected edge satisfies
    \begin{align*}
      \sphericalangle \{\POm_{k}(\tet) -\POm_{k+1}(\tet),
      \POmb_{k}(\tet) - \POmb_{k+1}(\tet)\}
      \le
      \frac{2\|\bfn\|_{C^0}}{\lOm_k(\tet)-2\|\bfn\|_{C^0}}\le
      {4\cb \,q\,\|\bfn\|_{C^0}}
    \end{align*}
    where in the last inequality we have
    used~\eqref{e_boundlE}: in fact, we know
    $\lOm_k(\tet) > \cb/q$ and by our assumptions on $\bfn$ we
    have $\|\bfn\|_{C^0}\le\|\bfn\|_{C^1} < c/q^8$, thus,
    if $c < 1/\cb$, since $q > 2$:
    \begin{align*}
      \lOm_k(\tet)-2\|\bfn\|_{C^0} \ge \lOm_k(\tet)/2 >
      1/(2\cb q).
    \end{align*}

    Since $\Om$ has an integrable rational caustic
    $\Gamma_{1/q}$ of rotation number $1/q$, the collection
    $\POm_k(\tet),\ k=0,\cdots,q-1$ corresponds to a
    $q$-periodic orbit, thus, the angle of incidence at
    $\POm_{k}(\tet)$ of $\POm_{k}(\tet) -\POm_{k+1}(\tet)$
    equals the angle of reflection of
    $P'_{k-1}(\tet) -P'_{k}(\tet)$.  See
    Figure~\ref{fig:reflect}: the angle between the tangent
    to $\partial \Om$ at $\POm_k(\tet)$ and the tangent to
    $\cE_e$ at the projected point $\POmb_k(\tet)$ is
    bounded above by $\bfn'(S_q(\tet+k/q))$, hence by
    $\|\bfn\|_{C^1}$.  Therefore, adding the two deviations
    coming from the discrepancy of the tangents to
    $\partial \Om$ (resp.\ $\cE_e$) and the discrepancy of
    end-points $\POm_i(\tet)$ (resp.\ $\POmb_i(\tet)$) with
    $i=k\pm 1,k$ we get that
    \begin{align*}
      |\phib^+_k-\phib^-_k|\le 4\cb\, q\, \|\bfn\|_{C^0}+ 2 \|\bfn\|_{C^1},
    \end{align*}
    from which we conclude our proof.
  \end{proof}
   \begin{figure}[h]
      \begin{center}
        \includegraphics{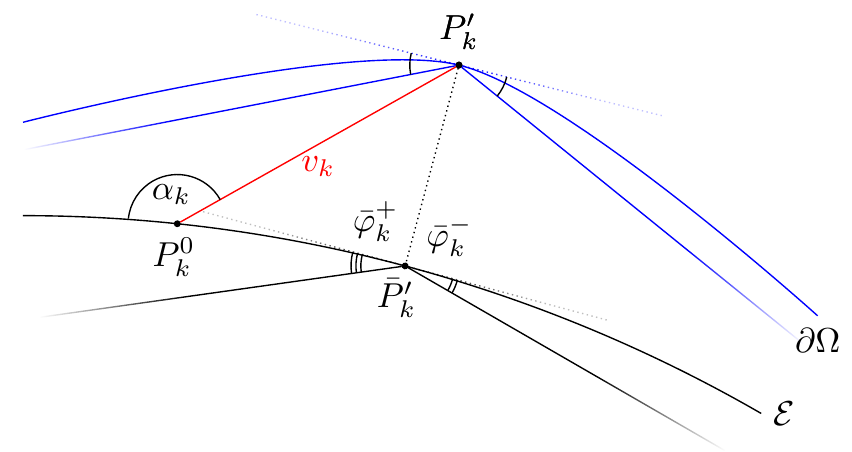}
      \end{center}
      \caption{ Reflection angles: in blue (above) the trajectory of
        the periodic orbit given by $\POm_0,\cdots,\POm_{q-1}$; in
        black (below) the pseudo-orbit given by $\POmb_0,\cdots,\POmb_{q-1}$.}
      \label{fig:reflect}
    \end{figure}
  \begin{lemma}
    For each $k=0,\cdots,q-1$ let $\bar\tet_k$ be so that
    $\POmb_k(\tet) = P_{S_q(\bar\tet_k)}$.  Then there
    exists $C=C(e,\|\bfn\|_{C^5})$ so that, in the above
    notations, {for any \ $k = 0,\cdots,q-1$}:
    \begin{align}
      \label{v-bound}
      |\bar\tet_k-\tet_k|&\le Cq^3\|\bfn\|_{C^1},&
      v_k(\tet)&\le Cq^3\|\bfn\|_{C^1}.
    \end{align}
  \end{lemma}
  \begin{proof} The basic idea of the proof is to consider
    the worst case scenario for the deviation of the
    reflection angles $\phib^{\pm}_k(\tet)$ from
    $\phi_k(\tet)$.  Since, unless $\cE_e$ is a circle, the
    reflection angles $\phi_k$ vary depending on the
    reflection point\footnote{  Reflection angles are
      smaller close to the end-points of the minor axis and
      larger close to the end-points of the major axis}, it
    is more convenient to keep track of a \emph{first
      integral}, which is constant along any orbit on the
    ellipse $\cE_e$ and, therefore, cannot change too
    rapidly for the perturbed domain $\Om$.  We now
    quantitatively explain this phenomenon.  Recall that for
    the ellipse one can explicitly define a conserved
    quantity (a first integral), as follows.  For
    simplicity, assume $\cE_e$ is centered at the origin and
    that the major axis is horizontal; let
    \begin{align*}
      \cE_e=\{x^2+y^2/(1-e^2)=a_e^2\},
    \end{align*}
    where $a_e$ is the semi-major axis, given by
    $a_e = 1/(4E(e))$, and $E(e)$ is the complete elliptic
    integral of the second kind, so that the ellipse $\cE_e$
    has, as we always assume, perimeter $1$.  Let us then
    introduce the so-called \emph{elliptical coordinates}
    $(\mu,\psi)$ on $\R^2$ as follows:
    \begin{align*}
      x&=h\cdot \cosh \mu \cdot \cos \psi,&
      y&=h\cdot \sinh \mu \cdot \sin \psi
    \end{align*}
    where
    $h^2=a_e^2e^2,\ 0\le \mu <\infty,\ 0\le \psi<2\pi$.  The
    family of co-focal ellipses $\mu=$const and hyperbolas
    $\psi=$const form an orthogonal net of curves\footnote{
      Observe that as $e\to 0$, we have $h\to 0$ and
      $\mu\to\infty$ so that $h\cosh\mu\to a_0$ and
      $h\sinh\mu \to a_0$, where $a_0 = 1/(2\pi)$.}.  The ellipse $\cE_e$ has the
    equation $\mu=\mu_0$, where $\cosh^2 \mu_0=e^{-2}>1$.
    Thus, the length parametrization $s$ of the ellipse can
    be given as a function of $\psi$, (see e.g. \cite{Ta}
    for an explicit formula): Then, the billiard map has a
    first integral given by
    \begin{align*}
      I(\psi,\phi)= \cos^2 \phi+ \frac{\cos^2 \psi}{\cosh^2 \mu_0}\sin^2
      \phi;
    \end{align*}
    observe that $I(\psi,\phi) = I(\psi,\pi-\phi)$.  Recall
    that $\tet$ denotes the action-angle parametrization of
    $\cE_e$ in action-angle coordinates with rotation number
    $1/q$ and $S_q$ is the change of variables to arc-length
    coordinates.  Since the elliptic angle $\psi$ is an
    analytic function of the arc-length parametrization $s$
    and $S$, in turn, is an analytic function of $\tet$ (see
    (\ref{e_dynamicalS})), we can define the first integral
    $I(\tet,\phi)$ in the $(\tet,\phi)$ coordinates. Notice
    that $\cosh^2 \mu_0> 1\ge\cos^2 \psi$; hence
    \begin{align*}
      \partial_\phi I(\psi,\phi) = \left(\frac{\cos^2 \psi}{\cosh^2 \mu_0}-1\right)\,\sin
      2\phi;
    \end{align*}
    observe that for any $\psi$, the function
    $I(\psi,\cdot)$ is strictly decreasing on $(0,\pi/2)$;
    moreover $|\partial_\phi I| < 1$ and
    \begin{align}\label{e_improvedPartialI}
      |\partial_\phi I|\in [1-\cosh^{-2} \mu_0,2]\, \phi\ \text{ for $\phi \in [0,\pi/6]$}.
    \end{align}
    Moreover, this holds in both $(\psi,\phi)$ and
    $(\tet,\phi)$ coordinates.

    Then we claim that there exists $\ksp$ so that
    $\phib_{\ksp}^- \le \phif_q(\bar\tet_\ksp) \le
    \phib_{\ksp}^+$.  Observe that by definition
    \begin{align*}
    f(S_q(\bar\tet_k),\phib_k^+) =
    (S_q(\bar\tet_{k+1}),\phib_{k+1}^-);
    \end{align*}
    by well-known properties of monotone twist maps, no
    orbit can cross the invariant curve $\Gm_{1/q}$, thus we
    obtain that if $\phib^+_k < \phif_q(\bar\tet_k)$
    (resp.\ $\phib^+_k > \phif_q(\bar\tet_k)$), then
    $\phib^-_{k+1} < \phif_q(\bar\tet_{k+1})$
    (resp.\ $\phib^-_{k+1} > \phif_q(\bar\tet_{k+1})$).  We
    conclude that if our claim does not hold, necessarily,
    either $\phib_k^+ < \phif_q(\bar\tet_{k})$ or
    $\phib_k^+ > \phif_q(\bar\tet_{k})$ for all
    $k = 0,\cdots,q-1$.  In the first case, the twist
    condition implies that
    $\bar\tet_{k+1} -\bar\tet_{k} < 1/q$; but this is a
    contradiction, since $\bar\tet_{q} = \bar\tet_{0}+1$
    (passing to the covering space $\R$).  Similar arguments
    in the second case also lead to a contradiction; this, in
    turn, implies our claim.  Moreover,
    Lemma~\ref{lem:deviated-angle} implies that
    \begin{align*}
      \phib^+_{\ksp}-\phif_q(\bar\tet_\ksp) \le
      5\cb \,q\,\|\bfn\|_{C^1} < 5q^{-7}.
    \end{align*}
    Define now the \emph{instant first integral}
    $I_k^\pm = I(\bar\tet_k,\phib_k^\pm)$; then
    $I^+_k = I^-_{k+1}$ and since
    \begin{align*}
      |I^+_k-I^-_k| \le
      \left|\int_{\phib_k^-}^{\phib_k^+}\partial_\phi I(\bar\tet_k,\phi)d\phi\right|
    \end{align*}
    and $\phif_q(\bar\tet_\ksp) < C(e)/q$ (applying
    Corollary~\ref{lazutkin-coordinates-curves} to $\cE_e$),
    Lemma~\ref{lem:deviated-angle}
    and~\eqref{e_improvedPartialI} allow us to conclude
    (possibly choosing a larger $C$) that
    \begin{align}\label{e_boundOnI}
      |I^+_\ksp-I_*|< C\,\|\bfn\|_{C^1},
    \end{align}
    where $I_* = I(\theta, \phi_0(\tet))$ and
    $C = C(e,\|\bfn\|_{C^5})$.  Inducing at most $q$ times
    and applying repeatedly the same argument we conclude
    that $|I_0^\pm-I_*| < Cq\|\bfn\|_{C^1}$.  This in turn
    implies that
    \begin{align*}
      |\phib_0^{\pm}(\tet)-\phi_0(\tet)| <Cq^2\|\bfn\|_{C^1}
    \end{align*}
    and inducing on $k$ and using again
    Lemma~\ref{lem:deviated-angle} we conclude (possibly
    choosing a larger $C$)
    \begin{align*}
      |\bar\theta_k-\tet_k|&<Cq^3\|\bfn\|_{C^1}.
    \end{align*}
    The second bound of~\eqref{v-bound} follows immediately
    by applying the triangle inequality.
  \end{proof}
  \begin{lemma}
    \label{prop:one-leg-perimeter}
    In the notations introduced above we have
    \begin{align}\label{e_one-leg-perimeter}
      \Big|\lOm_k(\tet)-\lE_k(\tet)&-v_k(\tet)\cos\left(\phi_k(\tet)+\al_k(\tet)\right)\\
                                   &+v_{k+1}(\tet)\cos\left(\phi_{k+1}(\tet)-\al_{k+1}(\tet)\right)\Big|\le 10\frac{v_k(\tet)^2+v_{k+1}(\tet)^2}{\lE_{k}(\tet)}.\notag
    \end{align}
  \end{lemma}
  \begin{proof}%
    Let $\pEm_k(\tet) = \|\POm_k(\tet)-\PE_{k+1}(\tet)\|$;
    applying the Cosine Theorem to the triangle
    $\bigtriangleup \PE_k(\tet) \PE_{k+1}(\tet)\POm_k(\tet)$
	we have
    \begin{align*}
      \pEm_k(\tet)^2&= v_k(\tet)^2+\lE_k(\tet)^2 - 2 v_k(\tet)\lE_k(\tet) \cos
                      (\phi_k(\tet)+\al_k(\tet)).
    \end{align*}
    Likewise, applying it to the triangle
    $\bigtriangleup \PE_{k+1}(\tet) \POm_{k+1}(\tet) \POm_k(\tet)$ we have
    \begin{align*}
      \lOm_k(\tet)^2 =
      v_{k+1}(\tet)^2+\pEm_k(\tet)^2 + 2 v_{k+1}(\tet)\pEm_k(\tet)\cos
      (\phi_{k+1}(\tet)-\al_{k+1}(\tet)-\dt_{k+1}(\tet)),
    \end{align*}
    where $\dt_{k+1}(\tet)$ is the oriented angle
    $\sphericalangle(\PE_k(\tet) \PE_{k+1}(\tet)
    \POm_k(\tet))$.
    Combining the above expressions we get
    \begin{align}\label{e_lastEquality}
      \lOm_k(\tet)^2-\lE_k(\tet)^2 &= v_k(\tet)^2 + v_{k+1}(\tet)^2-2v_k(\tet)\lE_k(\tet)\cos
                                     (\phi_k(\tet)+\al_k(\tet))\\&\phantom =
                                                                   +2v_{k+1}(\tet)\pEm_k(\tet) \cos
                                                                   (\phi_{k+1}(\tet)-\al_{k+1}(\tet)-\dt_{k+1}(\tet)).\notag
    \end{align}
    Observe that by the triangle inequality:
    \begin{align*}
      \lE_k(\tet)-v_k(\tet)-v_{k+1}(\tet)\le \lOm_k(\tet),\pEm_k(\tet)\le \lE_k(\tet)+v_k(\tet)+v_{k+1}(\tet).
    \end{align*}
    Moreover, elementary geometry implies
    $|\sin\dt_{k+1}(\tet)|\le v_k(\tet)/\lE_k(\tet)$.
    Now~\eqref{e_one-leg-perimeter} immediately follows
    dividing both sides of~\eqref{e_lastEquality} by
    $\lOm_k(\tet)+\lE_k(\tet)$ and using the above
    estimates.
  \end{proof}
  We can now conclude the proof of
  Theorem~\ref{perturbation-perimeter}; observe that by
  definition $\LE_q(\tet) = \sum_{k = 0}^{q-1}\lE_k(\tet)$ and
  likewise $\LOm_q(\tet) = \sum_{k = 0}^{q-1}\lOm_k(\tet)$.
  By Lemma~\ref{prop:one-leg-perimeter} we thus gather:
  \begin{align*}
    \Big|\LOm_q(\tet)-\LE_q(\tet)&-\sum_{k = 0}^{q-1}v_k(\tet)\cos\left(\phi_k(\tet)+\al_k(\tet)\right)\\
                                 &+\sum_{k =
                                   0}^{q-1}v_{k+1}(\tet)\cos\left(\phi_{k+1}(\tet)-\al_{k+1}(\tet)\right)\Big|\le 20\sum_{k = 0}^{q-1}\frac{v_k(\tet)^2}{\lE_k(\tet)}.
  \end{align*}
  Observe that
  \begin{align*}
    \sum_{k=0}^{q-1} \Big[&-v_k(\tet) (\cos \phi_k(\tet)\cos
                            \al_k(\tet)- \sin \phi_k(\tet)\sin \al_k(\tet))\\
                          &+ v_{k+1}(\tet) (\cos \phi_{k+1}(\tet)\cos \al_{k+1}(\tet)+ \sin
                            \phi_{k+1}(\tet)\sin \al_{k+1}(\tet))\Big]\\
                          &= 2\sum_{k=0}^{q-1} v_k(\tet)\sin \phi_k(\tet)\sin
                            \al_k(\tet).
  \end{align*}
  Notice that, by~\eqref{v-bound}, we have
  $v_k(\tet)\sin \al_k(\tet)=
  \bfn(S_q(\tet+k/q))+O(q^6\|\bfn\|^2_{C^1})$.  Therefore,
  \begin{equation*}
    \left|\LOm_{q}(\tet)-\LE_{q}(\tet)-\sum_{k=0}^{q-1}
    \bfn(S_q(\tet+k/q))\sin \phif_q(\tet+k/q) \right|
    \le Cq^8\|\bfn\|_{C^1}^2.
  \end{equation*}
  This completes the proof of Theorem~\ref{perturbation-perimeter}.
\end{proof}

\section{Lazutkin parametrization and Deformed Fourier
  Modes} \label{s_DeformedFourierModes}%
It turns out that for nearly glancing orbits, i.e.\ orbits
having small reflection angle, it is more convenient to
study the billiard map $f$, which has been defined
in~\eqref{eq:billiard-map}, in Lazutkin coordinates
(see~\cite{L}), which we now proceed to define.

Let $\Om$ be a strictly convex domain; recall that $s$
denotes the arc-length parametrization of $\partial \Om$ and
denote with $\rho(s)$ its radius of curvature at $s$.
Observe that if $\Om$ is $C^r$, then $\rho$ is $C^{r-2}$.
Define the \emph{Lazutkin parametrization} of the boundary:
\begin{align}
  \label{eq:Lazutkin}%
  x(s) &= C_\Om\,\int_0^s
         \,\rho(\sigma)^{-2/3}\ d\sigma, &\text{ where }C_\Om
  &=\left[\int_0^{\ell_{\partial\Om}}\rho(\sigma)^{-2/3}d\sigma\right]^{-1}.
\end{align}
We call \emph{the Lazutkin map} the following change of
variables:
\begin{align} \label{eq:Lazutkin-map}%
  \lazu&:(s,\varphi)\mapsto (\,x=x(s), y(s,\varphi)=4C_\Om
         \,\rho(s)^{1/3} \sin (\phi/2)\,).
\end{align}
Also let us introduce the \emph{Lazutkin density}
\begin{align}
  \label{correction-function}
  \mu(x)=\frac{1}{2C_\Om\rho(x)^{1/3}},
\end{align}
where we denote by
$\rho(x)=\rho(s(x))$ the radius of curvature in the Lazutkin
parametri-\,zation, where $s(x)$ can be obtained by
inverting~\eqref{eq:Lazutkin}.  Observe that
$\mu(x) =\pi$ for a circle and varies analytically with the
eccentricity for ellipses.

By replacing the arc-length parametrization $s$ with the
Lazutkin parametrization $x$ in the definition of the
tubular coordinates, we obtain the definition of the
\emph{Lazutkin tubular coordinates}.  With a slight abuse of
notation, we denote the corresponding perturbation function
with $\bfn(x)$.  Observe that if $\partial\Om = \cE_e$ is an
ellipse, $\rho$ is analytic and, thus, the Lazutkin
parametrization is itself an analytic parametrization of
$\cE_e$.
\begin{lemma}\label{l_approximationLemma}
  Let $\Om$ be a perturbation of the ellipse $\cE_e$
  identified by the function $\bfn$ (i.e.\
  $\partial\Om = \cE_e+\bfn$).  Consider another ellipse
  $\bar\cE$ sufficiently close to $\cE_e$: let
  $\bfn_{\bar\cE}$ so that $\bar\cE = \cE_e+\bfn_{\bar\cE}$
  and $(\bar x, \bar n)$ denote Lazutkin tubular coordinates
  in a neighborhood of $\bar\cE$.  If $\bar\cE$ is
  sufficiently close to $\cE_e$ we can write
  $\partial\Om = \bar\cE+\bar\bfn$ for some function
  $\bar\bfn(\bar x)$.  There exists $C = C(e)$ so that
  \begin{align}\label{e_pointwiseBound}
    |\bar\bfn(x)-(\bfn(x)-\bfn_{\bar\cE}(x))|\le C\|\bfn_{\bar\cE}\|_{C^{1}}\|\bfn-\bfn_{\bar\cE}\|_{C^1}.
  \end{align}
  In particular, for any $C' > 1$, if $\bar\cE$ is
  sufficiently close to $\cE_e$ we have
  \begin{align}\label{e_uniformBound}
    \frac1{C'}\|\bfn-\bfn_{\bar\cE}\|_{C^1} \le  \|\bar\bfn\|_{C^1} \le C'\|\bfn-\bfn_{\bar\cE}\|_{C^1}.
  \end{align}
\end{lemma}
\begin{proof}
  Consider the change of variables
  $(x,n)\mapsto (\bar x,\bar n)$ defined in the intersection
  of the tubular neighborhoods of $\cE_e$ and $\bar\cE$.
  Clearly this is an analytic change of variables, that is
  $C\|\bfn_{\bar\cE}\|_{C^0}$-close to the identity in any
  $C^r$-norm for some $C$ depending on $r$ and on the
  eccentricity $e$.  In particular, we have:
  \begin{align*}
    \bar x(x,n) &= x+ \varrho_1(x,n),\\
    \bar n(x,n) &= (n-\bfn_{\bar\cE}(x))(1+
                       \varrho_2(x,n)),
  \end{align*}
  where $\varrho_1$ and $\varrho_2$ are analytic functions
  that are $C\|\bfn_{\bar\cE}\|_{C^0}$-small in any
  $C^r$-norm for some $C$ depending on $r$ and on the
  eccentricity $e$.  Observe that if $x_\crit$ is a critical
  point of $\bfn_{\bar\cE}$, we have by construction
  $\bar n(x_\crit,n) = n-\bfn_{\bar\cE}(x_\crit)$.  Since
  $\partial\Om = \cE_e+\bfn = \bar\cE+\bar\bfn$, we conclude
  that
  \begin{align*}
    \bar n(x,\bfn(x)) = \bar\bfn(\bar x(x,\bfn(x))).
  \end{align*}
  Let us denote with $\bar x_\Omega(x) = \bar x(x,\bfn(x))$;
  observe that by our previous estimates we have that $\bar
  x_\Om$ is a diffeomorphism and $\bar x_\Om' = 1+O(\|\bfn_{\bar\cE}\|_{C^1}\|\bfn\|_{C^1})$.
  By the implicit function theorem we conclude that
  \begin{align*}
    \bar\bfn'(\bar x_{\Om}(x))) =%
    \frac{\partial_x \bar n(x,\bfn(x)) + \partial_{n}\bar n(x,{\bfn}(x)) \bfn'(x)}{\partial_{x} \bar x(x,{\bfn}(x))
    + \partial_{n}\bar x(x,{\bfn}(x)) \bfn'(x)}.
   \end{align*}
  Using the above expression for $n(\bar x,\bar n)$ and
  $x(\bar x,\bar n)$ we gather
  \begin{align*}
    \bar\bfn'(\bar x_\Om(x))) = \frac{
(\bfn'(x)-\bfn'_{\bar\cE}(x))
(1+O(\|\bfn_{\bar\cE}\|_{C^0}))}
{1+O(\|\bfn_{\bar\cE}\|_{C^0})}.
  \end{align*}
  Thus, integrating:
  \begin{align*}
    \bar\bfn(\bar x) &=
    \bar\bfn(x_\crit)+\int_{x_\crit}^{\bar x}\bar\bfn'(\bar
                       x)d\bar x\\
    &= \left[\bfn(x_\Om\inv(\bar
      x))-\bfn_{\bar\cE}(x_\Om\inv(\bar
      x))\right](1+O(\|\bfn_{\bar\cE}\|_{C^0}))\\
      & =\bfn(\bar x)-\bfn_{\bar\cE}(\bar x)+O(\|\bfn_{\bar\cE}\|_{C^1}\|\bfn-\bfn_{\bar\cE}\|_{C^1}),
  \end{align*}
that is~\eqref{e_pointwiseBound}.  It is then immediate to obtain~\eqref{e_uniformBound}.
\end{proof}

Consider now the billiard map in Lazutkin coordinates
$\fL = \lazu\circ f\circ\lazu^{-1}$; then $\fL$ has the
following form (see e.g.~\cite[(1.4)]{L}):
\begin{align}
  \label{Lazutkin-map}
  \fL : (x,y)&\to
               (x+y+y^3g(x,y),y+y^4h(x,y)),
\end{align}
where $g$ and $h$ can be expressed analytically in terms of
derivatives of the curvature radius $\rho$ up to order $3$:
hence, if $\Om$ is $C^r$, $g,h$ are $C^{r-5}$.  Recall that
$\widehat \Gamma_{1/q} \subset \Om$ denotes a caustic of
rotation number $1/q$, while $\Gamma_{1/q}$ denotes the
associated non-contractible invariant curve for the billiard
map $f$.  We denote by $\GmL_{1/q}$ the corresponding
invariant curve for the billiard map $\fL$ in Lazutkin
coordinates, i.e. $\GmL_{1/q} = \lazu\, \Gm_{1/q}$.
Moreover, let us introduce the change of variables from
action-angle coordinates $(\tet,\om)$ to Lazutkin
coordinates, i.e
$(X(\tet,\om),Y(\tet,\om))
=\lazu(S(\tet,\om),\phif(\tet,\om))$;
as before, we define $X_q(\tet) = X(\tet,1/q)$ and
$Y_q(\tet) = Y(\tet,1/q)$.
\begin{lemma}\label{lazutkin-coordinates-orbits}
  Let $\Om$ be a $C^5$ strictly convex domain; for
  $k\in \Z$, let $(x_k,y_k) = \fL^k(x_0,y_0)$ be a periodic
  orbit of rotation number $1/q$ with $q>2$.  Then there
  exists $C$ depending on $\|\rho\|_{C^3}$ and {independent
    of $q$}, such that for $0\le k < q$
  \begin{align}\label{e_estimateOnLazutkin-orbits}
    \left|y_k- \frac 1q \right|&< \frac {C}{q^3},&
   \left|\tilde
      x_k-\tilde x_0-\frac kq \right|&<\frac {C}{q^2},
  \end{align}
  where $\tilde x_k$ is a lift of $x_k$ to $\R$.
\end{lemma}
\begin{corollary}\label{lazutkin-coordinates-curves}
  Let $\Om$ be a $C^5$ strictly convex domain and let
  $\GmL_{1/q}$ be the invariant curve corresponding to an
  integrable rational caustic of rotation number $1/q$ with
  $q>2$, given by
  \begin{align*}
    \GmL_{1/q}=\{ (x,y_q(x)): \ x\in \T\}.
  \end{align*}
  Then there exists $C$ depending on $\|\rho\|_{C^3}$, such
  that
  \begin{align}\label{e_estimateOnLazutkiny-curves}
    \left|y_q(x)- \frac 1q \right|&< \frac {C}{q^3} &\text{ for any $\ x\in\T$.}
  \end{align}
  Moreover, in the case $\partial\Om$ is an ellipse $\cE_e$ of
  eccentricity $e$ and perimeter $1$, the constant $C$
  {can be chosen to depend continuously} on $e$
  and satisfies $C(e)\to0$ as $e\to0$.
\end{corollary}
\begin{proof}
  The proof of the first part immediately follows from
  the first bound of~\eqref{e_estimateOnLazutkin-orbits}.
  Observe now that if $\partial\Om$ is an ellipse of
  eccentricity $e$,
  $\GmL_{1/q} =\{(X_q(\tet),Y_q(\tet))\}_{\tet\in\T}$ where
  both $X_q$ and $Y_q$ vary analytically with $e$.
  Moreover, if $\partial\Om$ is a circle, $Y_q(\tet)$ is the
  constant function equal to $1/q$.  We conclude that we can
  choose $C(e)$ so that {it is continuous in $e$ and} $\lim_{e\to0}C(e) = 0$.
\end{proof}
\begin{corollary}
  \label{edge-uniform}
  Let $\Om$ be a $C^5$ strictly convex domain and
  $q > 2$.  Let $(s_k,\varphi_k),\ k=0,\cdots,q-1$ be a
  $q$-periodic orbit of rotation number $1/q$ and
  $P_k, \ k=0,\cdots,q-1$ be the corresponding collision
  points on $\partial \Om$. Then there is $\cb=\cb(\Om)>1$,
  depending on $\|\rho\|_{C^3}$, such that the Euclidean
  length of each edge $\|P_{k+1} -P_{k}\|$ satisfies
  \begin{align*}
    \frac{1}{\cb q} \le \|P_{k+1} -P_{k}\| \le \frac{\cb}{q}.
  \end{align*}
  Moreover, if $\Om$ is a perturbation $\bfn$ of an ellipse
  $\cE_e$ (i.e. $\partial\Om = \cE_e+\bfn$), then $\cb$
  {can be chosen to} depend continuously on the
  eccentricity $e$ and $\|\bfn\|_{C^5}$.
\end{corollary}
\begin{proof}
  Recall that, by definition,
  $y(s,\phi)=4 \,C_\Om \,\rho^{1/3}(s) \sin (\phi/2)$.  By
  Lemma \ref{lazutkin-coordinates-orbits} we have
  $y \in [1/q-C/q^3,1/q+C/q^3]$ for some $C$ depending on
  $\rho$ only. Therefore,
  $\sin (\phi/2)\in [1/Cq-1/q^3,C/q+C^2/q^3]$.  Since the
  angle of reflection is of order $1/q$ and curvature is
  uniformly bounded, we get the required bound on the
  distance $\|P_{k+1}-P_{k}\|$.
\end{proof}
\begin{proof}[{Proof of Lemma~\ref{lazutkin-coordinates-orbits}}]
  Choose $q_0$ (sufficiently large depending on
  $\|\rho\|_{C^3}$) to be specified in due course and assume
  $q\ge q_0$.  {Observe that we can choose $C$ so
    large that our statement trivially holds for any
    $2 < q < q_0$.}  First of all, we claim that we have the
  preliminary bound
  \begin{align*}
    y_k&\le \frac{C_1}{q},&\text{ for } 0\le k <q,
  \end{align*}
  where $C_1$ is a large constant depending on the curvature
  $\rho$.  In fact,
  let $(s_k,\varphi_k) = \lazu^{-1}(x_k,y_k)$, so
  that
  \begin{align*}
  (s_{k+1},\varphi_{k+1}) = f(s_k,\varphi_k)
  \end{align*}
  and let $\tilde s_k$ be a lift to $\R$.  Since
  $\tilde s_q = \tilde s_0 + 1$, there exists
  $0\le \ksp < q$ so that
  $0<\tilde s_{\ksp+1}-\tilde s_\ksp\le 1/q$.  For fixed
  $s_k$, we can find a function $\varphi(s_{k+1})$ so that
  the ray leaving $s_k$ with angle $\varphi(s_{k+1})$ will
  collide with $\partial\Om$ at $s_{k+1}$; if $q_0$ is
  sufficiently large, we can use expansion of the billiard
  map for small $\varphi$ in terms of curvature (see
  e.g. \cite[(1.1)]{L}) and conclude that
  $\varphi_\ksp < C/q$, where $C = C(\|\rho\|_{C^1})$ and
  thus, by definition of the Lazutkin coordinate
  map~\eqref{eq:Lazutkin-map} we conclude that
  $y_\ksp\le {C_1}/{q}$, where
  $C_1 = C_1(\|\rho\|_{C^1})$.  By
  iterating~\eqref{Lazutkin-map}, starting from $\ksp$, we
  conclude by (finite) induction that for any $0\le k < q$:
  \begin{align*}
    |y_{k+1}-y_k|&\le \dfrac{C_0}{q^4}, & y_k &< \dfrac{C_1}q,
  \end{align*}
  where $C_0=\max \{\|g\|,\|h\|\}C_1^4$ and we have possibly
  chosen a larger $C_1$.  Observe that since $\|g\|$ and
  $\|h\|$ depend on $\|\rho\|_{C^3}$, so does $C_0$.
  Moreover, by iterating the first inequality $q$ times we
  also have
  \begin{align}\label{e_closeyq}
    |y_j-y_k|&\le \dfrac{C_0}{q^3}&\text{ for any } 0\le k,j <q.
  \end{align}
  We now claim that $|y_k-1/q| \le 4C_0/q^3$ for any
  $0\le k < q$.  Assume by contradiction that for some $j$,
  $y_{j}-1/q > 4C_0/q^3$.  Then by~\eqref{e_closeyq} we
  gather that $y_k-1/q > 3C_0/q^3$ for any $0\le k < q$.
  Hence, by~\eqref{Lazutkin-map} and the above estimates,
  for any $0\le k < q$ we have, assuming $q_0$ is
  sufficiently large:
  \begin{align*}
    \tilde x_{k+1}-\tilde x_k&\ge \frac 1q + \frac{C_0}{q^3}.
  \end{align*}
  Iterating $q$ times, we conclude that
  \begin{align*}
  \tilde x_q -\tilde x_0\ge 1+\frac{C_0}{q^2},
  \end{align*}
  which is a contradiction, since
  $\tilde x_q = \tilde x_0 +1.$ A similar argument implies
  that if there exists $0 \le j < q$ so that
  \begin{align*}
    y_j-\frac 1q < -\frac{4C_0}{q^3}
  \end{align*}
  we also reach a contradiction.  This implies our claim,
  which in turn
  implies~\eqref{e_estimateOnLazutkin-orbits}.  Notice that
  in order to have $C_0/q^3$ to be small compared to $1/q$
  we need $q_0$ (and thus $q$) to be sufficiently large
  (with respect to $\|\rho\|_{C^3}$).
\end{proof}
\begin{lemma}
  \label{action-angle-variation} Let $\cE_e$ be an ellipse of
  eccentricity $e$ and perimeter $1$; then there exists
  $C(e)$ with $C(e)\to 0$ as $e\to 0$ so that
  \begin{align*}%
    \|X_q-\Id\|_{C^1}\le\dfrac{C(e)}{q^2}.
  \end{align*}
\end{lemma}
\begin{proof}
  In the proof of this statement, to simplify the notation,
  $C(e)$ will denote an arbitrary constant which depends on $e$
  only; its actual value might change from an instance to
  the next.  Recall that $X(0,\om)$ parametrizes a fixed
  point $P_0$ (i.e. one of end points of the major axis) for
  all $\om\in [0,1/3]$.  Now consider the $q$-periodic orbit
  leaving the point $P_0$: in angle coordinates the orbit is
  given by
  $$
  \{\theta_k = k/q\mod1\}.
  $$
  Then by~\eqref{Lazutkin-map} and the definition
  of $(X_q(\theta),Y_q(\theta))$ we have
  \begin{align*}
    f_L(X_q(\theta),Y_q(\theta))=(X_q(\theta+1/k),Y_q(\theta+1/k)).
  \end{align*}
 and
 \begin{align*}
   X_q(\tet_{k+1})-X_q(\tet_k)=Y_q(\tet_k)\left(1+Y^2_q(\tet_k)\,
   g(X_q(\tet_k),Y_q(\tet_k)) \right).
 \end{align*}
 By Corollary~\ref{lazutkin-coordinates-curves} we conclude
 that
  \begin{align*}
    \left|\frac{X_q(\tet_{k+1})-X_q(\tet_k)}{\tet_{k+1}-\tet_{k}}-1\right|
    \le \frac{C(e)}{q^2};
  \end{align*} by the intermediate value theorem we conclude that
  there exists some $\bar{\tet}_k\in(\tet_k,\tet_{k+1})$ so
  that $|X_q'(\bar\tet_k)-1| < C(e)/q^2$.
%
%
Likewise, $|\bar\tet_k-\bar\tet_{k+1}|\le 2/q$ and
we can find $\bbar\tet_k\in(\bar\tet_k,\bar\tet_{k+1})$ so that
  $|X''_{q}(\bbar\tet_k)| \le C(e)/q$.  Hence, for each
  $\tet\in[\bar\tet_k,\bar\tet_{k+1}]$ we can write
  \begin{align*} X'_q(\tet) =
    X'_q(\bbar\tet_k)+\int_{\bbar\tet_k}^{\tet}\left[
    X''_q(\bbar\tet_k)+\int_{\bbar\tet_k}^{\tet'}X'''_q(\tet'')d\tet''\right]
    d\tet'.
  \end{align*} Now recall that $X_q(\tet) = S(\tet,1/q)$,
  where $S$ is analytic in both arguments; in particular, all
  derivatives of $X_q$ are bounded uniformly in $q$.  Moreover,
  $\|X'''_q\| < C(e)$ such that $C(e)\to0$ as $e\to0$,
since, as noted before, $X_q$ depends analytically on $e$
and for $e = 0$ the function $X_q$ is the identity.

  We conclude that $|X'_q(\tet)-X'_q(\bbar\tet)| < C(e)/q^2$
  for any $\tet\in[\bar\tet_k,\bar\tet_{k+1}]$, which implies
  that $\|X_q'-1\|_{C^0} < C(e)/q^2$.  Our estimate then holds
  integrating in $\tet$.
\end{proof}%
We now finaly proceed to define the functions
$\{c_q(x), s_q(x)\}_{q > 2}$ which we hinted at in
Section~\ref{s_strategy}.  Although the definition of such
functions can be carried out for an arbitrary convex domain
$\Om_0$, let us restrict ourselves to the case
$\partial\Om_0 = \cE_e$, for which they enjoy stronger
properties which are crucial for our later construction.
Recall that $s(x)$ denotes the length parametrization of
$\partial\Om_0$ as a function of the Lazutkin
parametrization, which can by obtained by
inverting~\eqref{eq:Lazutkin}.  Since
$y=4 \,C_\Om \,\rho(s)^{1/3} \sin (\phi/2)$, for any
$(s,\phi)\in
\Gm_{1/q}$,~\eqref{e_estimateOnLazutkiny-curves}
implies that:
\begin{align*}%
  \left| \sin \phif_q\left(X_q^{-1}(x)\right) -
  \frac{\conq}{2C_{\Om_0} q\rho(x)^{1/3}}
 \right| \le
  \frac{2C}{q^3},
\end{align*}
where $\conq = q\sin(\pi/q)/\pi\in[1/2,1]$.  Also,
Corollary~\ref{lazutkin-coordinates-curves} implies that, in
the above expression, $C = C(e)\to 0$ as $e\to0$.  To
simplify our notations let us introduce the auxiliary
function $\yq(x) = \sin \phif_q\left(X_q^{-1}(x)\right)$ and
notice, moreover, that $q\yq(x)$ has a well defined limit as
$q\to \infty$.  Recall that in~\eqref{correction-function}
we defined the Lazutkin Density
$\mu(x) = 1/(2C_{\Om_0}\rho(x)^{1/3})$.  Recall that the
density function $\mu(x)$ given above, depends only on the
domain {$\Om_0$} (i.e.\ on the eccentricity $e$); in
particular, it does not depend on $q$.  Using the previous
bound we have
\begin{align}
  \label{correction-deviation}
  \left|\frac{q\yq(x)}{w_q\mu(x)}-1\right|\le \frac{C}{q^2}
\end{align}
for some $C$ depending on $C  _{\Om_0}$ and $\rho$.  For any $q>2$
define\footnote{ We will define the first five functions
  $c_0(x),\, c_i(x), s_i(x), i=1,2$ respectively in the
  next section.}
\begin{subequations}\label{dynam-basis}
  \begin{align}
    c_q(x)&=\frac{q\yq(x)}{\conq\mu(x)}\
            \frac{1}{X'_q(X_q^{-1}(x))}\ \cos \, 2\pi qX_q^{-1}(x),
    \\
    s_q(x)&=\frac{q\yq(x)}{\conq\mu(x)}\
            \frac{1}{X'_q(X_q^{-1}(x))}\ \sin  \,2\pi qX_q^{-1}(x).
  \end{align}
\end{subequations}
Observe that Lemma~\ref{action-angle-variation} implies that
the above functions tend to the corresponding Fourier Modes
as $q\to\infty$.  We will henceforth refer to them as the
\emph{Deformed Fourier Modes}.  The next lemma gives a bound
on the speed of this approximation.
\begin{lemma}
  \label{error-estimate}
  Let $\cE_e$ be an ellipse of eccentricity $e$ and
  perimeter $1$; there exists $C^*(e)$ with $C^*(e)\to0$ as
  $e\to0$ so that for any $q>2$,
  \begin{align*}
    \|s_q-\sin (2\pi q\,\cdot)\|_{C^0}< \frac{C^*(e)}{q},\quad
    \|c_q-\cos (2\pi q\,\cdot)\|_{C^0}
    &< \frac{C^*(e)}{q}.
  \end{align*}
\end{lemma}
\begin{proof}
  By~\eqref{correction-deviation} and the bound of
  Lemma~\ref{action-angle-variation} we obtain
  \begin{align*}
    \left\|\frac{q\yq(x)}{\conq\mu(x)}\
            \frac{1}{X'_q(X_q^{-1}(x))}-1\right\|_{C^0} &< \frac{C^*(e)}{q^2};
    \intertext{likewise, Lemma~\ref{action-angle-variation} gives}
    \|\sin 2\pi qX_q^{-1}(x)-\sin 2\pi qx\|_{C^0} &<
                                    \frac{C^*(e)}q\\
    \|\cos 2\pi qX_q^{-1}(x)-\cos 2\pi qx\|_{C^0} &< \frac{C^*(e)}q
  \end{align*}
  from which we conclude.
\end{proof}
\begin{lemma}
  \label{l_projection-estimate}%
  Using the notations of Theorem
  \ref{perturbation-perimeter}, let $\cE_e$ be an ellipse of
  perimeter $1$ and eccentricity $e$ and $\partial\Om$ be a
  perturbation of $\cE_e$ identified by a $C^5$-smooth
  function\footnote{ Recall that we abuse notation and we
    also denote with $\bfn$ the perturbation as a function
    of the Lazutkin coordinate $x$; observe that since the
    change of variable is analytic, norms in arc-length and
    Lazutkin parametrization differ by some constant
    depending on $e$.} $\bfn(x)$; assume that $\Om$ has an
  integrable rational caustic $\Gm_{1/q}$ of rotation number
  $1/q$ for some $2 < q < c(e)\|\bfn\|_{C^1}^{-1/8}$. Then
  there exists $C=C(e,\|\bfn\|_{C^5}) > 0$ so that:
  \begin{align*}
    \left| \int \bfn(x) \mu(x) a_q(x) dx \right|
    \le Cq^8\|\bfn\|_{C^1}^2,
  \end{align*}
  where $a_q=c_q$ or $s_q$.
\end{lemma}
\begin{proof}
  Denote $\cD(\tet) = [\cD(\bfn,S,\phif;1/q)](\tet)$ the
  Deformation Function given by~\eqref{deformed-function};
  then by definition we have
  \begin{align*}
    \int_0^1\cD(\tet)\sin(2\pi q\tet)\,d\tet &= 
    2q\int_0^1\ \bfn\left(X_q(\tet) \right) \sin \phif_q\left(\tet \right)\,\sin (2\pi q\tet)\,d\tet\\ &= 
    2\int_0^1 \bfn\left(X_q(\tet) \right) \,[q\yq(X_q(\tet))]\,\sin (2\pi q\tet)\,d\tet.
  \end{align*} Notice that if $\Om$ has an integrable
	rational caustic $\Gm_{1/q}$ of a rotation number $1/q$
  for some $q>2$, then, using the notation introduced in
  Theorem~\ref{perturbation-perimeter}, perimeters
  $\LE_q(\tet)$ and $\LOm_q(\tet)$ of the $q$-gons inscribed
  in $\cE$ and $\partial\Om$, respectively, are constant.
  Therefore, Theorem~\ref{perturbation-perimeter} implies
  that the Deformation Function $\cD(\theta)$ is
  $Cq^8\|\bfn\|_{C^1}^2$ close to a constant.  Since, for
  any $k$,
  $\int_{k/q}^{(k+1)/q} \sin (2\pi q\tet)\,d\tet=0$, we
  conclude that
  \begin{align*}
    \left|\int_0^1\cD(\tet)\sin(2\pi q\tet)\,d\tet\right| &\le  Cq^8\|\bfn\|^2_{C^1}
  \end{align*}
  On the other hand, let us rewrite
  $x=X_q(\tet),\ \tet=X^{-1}_q(x)$: we obtain:
  \begin{align*}
    &\int_0^1 \bfn (x)\, \left[q\yq(x)\right]\,
      \sin (2\pi qX_q^{-1}(x))\ dX_q^{-1}(x)\\ &= \conq\int_0^1 \bfn(x)\ \mu(x)\
                                                 \frac{q\yq(x)}{\conq\mu(x)}\ \frac{1}{X'_q(X_q^{-1}(x))}\
                                                 \sin (2\pi
                                                 qX_q^{-1}(x))\,
                                                 dx\\
    &=\conq\int_0^1\bfn(x)\mu(x)s_q(x) dx,
  \end{align*}
  which gives the required inequality for $s_q$.  Repeating
  the argument verbatim, replacing $\sin(2\pi q\tet)$ with
  $\cos(2\pi q\tet)$ gives the corresponding inequality for
  $c_q$; this concludes the proof.
\end{proof}
\begin{lemma} \label{coeff-decay-simpler}%
  Let $\bfn(x)$ be a $C^{1}$ function, $\cE_e$ be an ellipse
  of eccentricity $e$ and perimeter $1$.  Then there is
  $C=C(e)>0$ such that for each $q > 2$ we have
  \begin{align*}
    \left|\int \bfn(x) \mu(x)c_q(x)dx\right|&\le
    \frac{C\|\bfn\|_{C^\rcloseness}}{q},&
    \left|\int \bfn(x) \mu(x)s_q(x)dx\right|&\le
    \frac{C\|\bfn\|_{C^\rcloseness}}{q}.
  \end{align*}
\end{lemma}
\begin{remark}
  In the above lemma, $C(e)$ does not tend to $0$ together
  with $e$.
\end{remark}
\begin{proof} Using Lemma~\ref{error-estimate} we have
  \begin{align*}
    \left|\int \bfn(x) \mu(x)c_q(x)dx - \int \bfn(x)
    \mu(x)\cos(2\pi qx)dx\right|\le \frac{C^*(e)\|\bfn\mu\|_{C^0}}q.
  \end{align*}
  Since $\mu(x)$ is analytic, the function $\bfn(x) \mu(x)$
  is $C^1$-smooth; hence, its $q$-th Fourier cosine
  coefficient satisfies the inequality
  \begin{align*}
    \left|\int \bfn(x)\, \mu(x)\,\cos (2\pi qx)dx\right| \le \frac{c\|\bfn\mu\|_{C^\rcloseness}}q
  \end{align*}
  This implies the required estimate, since $\|\mu\|_{C^1}$
  is bounded; the estimate for $s_q$ is completely analogous
  and it is omitted.
\end{proof}
\section{Selection of functional directions preserving the
  family of ellipses}\label{selected-five}
In this section we introduce the remaining $5$ Deformed
Fourier Modes, which we denote with $c_0,c_1,s_1,c_2,s_2$.
As in the case of the circle (see
Remark~\ref{r_motionDescription}), these five functions
generate homotheties ($c_0$), translations ($c_1,s_1$) and
hyperbolic rotations about an arbitrary axis ($c_2,s_2$).

In principle, we could define these functions for an
arbitrary smooth convex domain $\Om_0$.  We
refrain\footnote{ The reader could trivially modify our exposition
  and adapt it to the more general case.} to do so and
assume $\Omega_0$ is an ellipse, since all remaining
Deformed Fourier Modes have been defined only for ellipses.
To further fix ideas, assume that $\Omega_0 = \cE_e$ is
centered at the origin $O\in\R^2$ and that its major axis is
horizontal.  As usual, we assume that $\cE_e$ has perimeter
$1$.  


Let $(r,\ophi)$ denote polar coordinates on the plane; we
refer to $\{(r,\ophi): r\ge 0, \ophi = 0\}$ as the
\emph{polar axis}.  Let $r_e(\ophi)$ be the polar equation
of the ellipse $\cE_e$, i.e.
$\cE_e = \{(r_e(\ophi),\ophi)\,:\,\ophi\in\T\}$; let $x$ be
the Lazutkin parametrization of $\cE_e$ {so that
  $x = 0$ corresponds to the point} $(r_e(0),0)$.  Let
$x(\ophi)$ be the corresponding change of variable and
$\ophi(x)$ denote its inverse; observe that $x(\ophi)$ is
an analytic diffeomorphism. {Let $\theta^\trl_e(\ophi)$ be
the angle between the polar axis and the outward normal
to $\cE_e$ at $(r_e(\ophi),\ophi)$, measured in the
counter-clockwise direction.}  The function $\theta^{\trl}_e(\ophi)$
is strictly increasing and has topological degree $1$ by
the strict convexity of $\cE_e$. We gather that $\theta^\trl_e$
is an (analytic) diffeomorphism.  Moreover, $\theta^\trl_e$
depends analytically on $e$ and
$\|\theta^{\trl}_e-\Id\|_{C^\rcloseness}\to0$ as $e\to 0$.
Naturally, all functions on $\cE_e$ can be expressed with
respect to either the $\ophi$-\hspace*{0pt}parametrization
or the $x$-\hspace*{0pt}parametrization and differ via an
analytic change of variable; in particular, we let, with an
abuse of notation,
$\theta^{\trl}_e(x):=\theta^{\trl}_e(\ophi(x)),\
r_e(x):=r_e(\ophi(x))$.

We now fix $0\le e < 1$ and, in order to ease our notation,
let us drop $e$ from all subscripts.

Consider the ellipse $\cE^\hth\depp{a_0}$ obtained by replacing the
radial component $r(\ophi)$ with $\exp(a_0)r(\ophi)$ and
denote with $\bfn^\hth\depp{a_0}$ the corresponding perturbation
function so that $\cE^\hth\depp{a_0} = \cE+\bfn^\hth\depp{a_0}$.
Let us define the $0$-th Deformed Fourier mode as
\begin{align*}
    c_0(x)&:=r(x) \cos (\theta^\trl(x)-\ophi(x)).
\end{align*}
Observe that $\theta^\trl(x)-\ophi(x)$ is the angle
(measured in the counter-clockwise direction) between the
radial direction and the outer normal to $\cE$ at the point
identified by $x$.
\begin{lemma}\label{lm:homothety}%
  For $C$ depending on the eccentricity $e$ we have
  \begin{align*}
    \|\bfn^\hth\depp{a_0}- a_0\,c_0\|_{C^\rcloseness}\le Ca_0^2.
  \end{align*}
\end{lemma}
Similarly, for any (Cartesian) vector $(a_1,b_1)$, consider the ellipse
$\cE^\trl\depp{a_1,b_1}$ obtained by translating $\cE$ by $(a_1,b_1)$ and denote with
$\bfn^\trl\depp{a_1,b_1}$ the corresponding perturbation
function.  Let us define the first and second Deformed
Fourier modes as:
\begin{align*}
    c_1(x) &:= \cos(\theta^\trl(x)),& s_1(x) &:= \sin(\theta^\trl(x)).
  \end{align*}%
\begin{lemma}
  \label{lm:translation}
  For $C$ depending on the eccentricity $e$ we have:
  \begin{align*}
  \|\bfn^\trl\depp{a_1,b_1}-a_1c_1-b_1s_1\|_{C^\rcloseness}\le C
  (a_1^2+b_1^2).
  \end{align*}
\end{lemma}
Finally, let $\cE^\ecc\depp{a_2,b_2}$ be the ellipse
obtained by applying to $\cE$ the hyperbolic rotation
generated by the linear map
\begin{align}
  \label{hyperb-rotation}
  L\depp{a_2,b_2} = \exp
  \matr {a_2}{b_2}{b_2}{-a_2}.
\end{align}
Observe that the eccentricity $e^\ecc\depp{a_2,b_2}$ of the
ellipse $\cE^\ecc\depp{a_2,b_2}$ satisfies
$|e^\ecc\depp{a_2,b_2}-e|\le C\sqrt{a_2^2+b^2_2}$, where
$C=C(e)$.  Let $\bfn^\ecc\depp{a_2,b_2}$ be the
corresponding perturbation function and define
$\tet^\ecc(\ophi) := (\theta^\trl(\ophi)+\ophi)/2$; observe
that $\tet^\ecc$ is an analytic diffeomorphism satisfying
$\|\tet^\ecc-\Id\|_{C^\rcloseness}\to 0$ as $e\to0$.  Once
again we abuse notation and write $\tet^\ecc(x)$ for
$\tet^\ecc(\ophi(x))$; we can then define the third and
fourth Deformed Fourier mode as:
\begin{align*}
  c_2(x)&:=r(x)\cos 2\tet^\ecc(x)& s_2(x)&:=r(x)\sin 2\tet^\ecc(x).
\end{align*}
\begin{lemma}\label{lm:hyperbolic}
  For $C$ depending on the eccentricity $e$ we have
  \begin{align*}
    \|\bfn^\ecc\depp{a_2,b_2}- a_2 c_2-b_2s_2\|_{C^\rcloseness}\le C
    (a_2^2+b_2^2).
  \end{align*}
\end{lemma}
\begin{proof}[Proofs of
  Lemmata~\ref{lm:homothety}--\ref{lm:hyperbolic}]
  The proofs follow from elementary geometry and are left to
  the reader.
\end{proof}
\begin{corollary}\label{c_combiningLemmata}
  Let $\cE$ be an ellipse of eccentricity $e$ and perimeter
  $1$ and let $\bfn$ be a linear combination of $c_0$,
  $c_1$, $s_1$, $c_2$ and $s_2$, i.e.
  \begin{align*}
    \bfn= a_0c_0+a_1c_1+b_1s_1+a_2c_2+b_2s_2
  \end{align*}
  for some $a_0,a_1,b_1,a_2,b_2$ which we assume to be
  sufficiently small.  Then there exists $C$ depending on
  the eccentricity $e$ and an ellipse $\bar\cE$ so that
  $\bar\cE = \cE+ \bfn_{\bar\cE}$ with
  \begin{align*}
    \|\bfn - \bfn_{\bar\cE}\|_{C^{\rcloseness}} \le C  \| \bfn \|^2_{C^{\rcloseness}}.
  \end{align*}
\end{corollary}
\begin{proof}
  Let $\Om$ be so that $\partial\Om = \cE+\bfn$; denote with
  $\cE^* = \cE^\hth\depp{a_0}$ the ellipse obtained by
  applying to $\cE$ the homothety by $\exp(a_0)$ and let
  $\bfn_{\cE^*} = \bfn^\hth\depp{a_0}$.  By
  Lemma~\ref{lm:homothety} we have
  $\|\bfn_{\cE^*}-a_0c_0\|_{C^\rcloseness} < Ca_0^2$.  Let
  $\bfn^*$ be so that $\partial\Om = \cE^*+\bfn^*$; then by
  Lemma~\ref{l_approximationLemma} we gather that
  $\|\bfn^*-(\bfn-\bfn_{\cE^*})\|_{C^{\rcloseness}} <
  Ca_0\|\bfn-\bfn_{\cE^*}\|_{C^{\rcloseness}}$.
  Combining with the above estimate and by definition of
  $\bfn$ we conclude that
  $$
	\|\bfn^*-(a_0c_0+a_1c_1+b_1s_1+a_2c_2+b_2s_2)\|\le
  C\|\bfn\|^2_{C^\rcloseness}.
$$
  Let $c^*_q$ and $s^*_q$ denote the Deformed Fourier modes
  for $\cE^*$; then by construction we have
  $\|c^*_q-c_q\|_{C^\rcloseness} < Ca_0$ (and similarly for
  $s^*_q-s_q$) for $q = 1,2$.  We conclude that:
  \begin{align*}
    \|\bfn^*-(a_1c^*_1+b_1s^*_1+a_2c^*_2+b_2s^*_2)\|\le C\|\bfn\|^2_{C^\rcloseness}.
  \end{align*}
  Now let $\cE^{**} = \cE^{*\trl}\depp{a_1,b_1}$ be the
  ellipse obtained by applying to $\cE^*$ the translation by
  the vector $(a_1,b_1)$ and let
  $\bfn^*_{\cE^{**}} = \bfn^{*\trl}\depp{a_1,b_1}$; by
  Lemma~\ref{lm:translation} we have
  $$
\|\bfn^*_{\cE^{**}}-(a_1c^*_1+b_1s^*_1)\|\le C(a_1^2+b_1^2).
$$
  Let $\bfn^{**}$ be so that $\partial\Om = \cE^{**}+\bfn^{**}$
  and let $c_q^{**}$ and $s_q^{**}$ denote the Deformed
  Fourier modes for $\cE^{**}$; then arguing as before we
  conclude that
  \begin{align*}
    \|\bfn^{**}-(a_2c^{**}_2+b_2s^{**}_2)\|\le C\|\bfn\|^2_{C^{\rcloseness}}.
  \end{align*}
  Finally, let $\bar\cE = \cE^{**\ecc}\depp{a_2,b_2}$ be
  the ellipse obtained by applying to $\cE^{**}$ the
  hyperbolic rotation $L\depp{a_2,b_2}$ and let
  $\bfn^{**}_{\bar\cE} = \bfn^{**\ecc}\depp{a_2,b_2}$; by
  Lemma~\ref{lm:hyperbolic} we have
  $$
\|\bfn^{**}_{\bar\cE}-(a_2c_2^{**}+b_2s_2^{**})\|_{C^\rcloseness}\le
  C(a_2^2+b_2^2).
$$
Let $\bar\bfn$ be so that $\partial\Om =
  \bar\cE+\bar\bfn$; arguing once again as before, we
  conclude that $\|\bar\bfn\|_{C^\rcloseness}\le
  C\|\bfn\|_{C^\rcloseness}^2$, which then concludes our
  proof by means of Lemma~\ref{l_approximationLemma}.
\end{proof}

\begin{remark*} The norm $\|\cdot\|_{C^\rcloseness}$ in all
  previous estimates could in fact be replaced with the norm
  $\|\cdot\|_{C^r}$ for any $r\ge 0$, since all involved
  quantities are analytic functions.
\end{remark*}
  We can now extend
  Lemma~\ref{error-estimate}:
\begin{lemma}\label{error-estimate-2}
  In the notation of Lemma~\ref{error-estimate} and possibly
  increasing $C^{*}(e)$, for any positive integer $q$ we
  have
  \begin{align*}
    \|c_0-1\|_{C^0}&\le C^{*}(e),&
    \|c_q-\cos(2\pi q\cdot)\|_{C^0}&\le \frac{C^*(e)}q,&
    \|s_q-\sin(2\pi q\cdot)\|_{C^0}&\le \frac{C^*(e)}q.
  \end{align*}
\end{lemma}
\begin{proof}
  The case $q>2$ is covered by Lemma~\ref{error-estimate}.
  The cases $q=0,1,2$ follow by the above definitions.
\end{proof}
From now on, for convenience of notation we rename and
normalize the functions $c_q$ and $s_q$ as follows: let
$e_0 = c_0$ and for $j > 0$ let $e_j$ so that
$e_{2j}=\sqrt2\,c_j$ and $e_{2j-1}=\sqrt2\,s_{j}$.  The five
functions that we introduced in this section generate
deformations which preserve integrability of all rational
caustics, as the following lemma shows.
\begin{lemma}\label{l_orthogonality}
  Let $0\le j\le 4$ and $k > 4$; then
  \begin{align*}
    \int e_j(x)\mu(x)e_k(x)dx = 0.
  \end{align*}
\end{lemma}
\begin{proof}
  For any $\eps > 0$ small, consider the $\eps$-deformation
  of the ellipse $\cE_e$ {identified by $\bfn = \eps e_j$.  By
    Lemmata~\ref{lm:homothety}--\ref{lm:hyperbolic} there
    exists another ellipse $\bar\cE$ so that
    $\bar\cE = \cE+\bfn_{\bar\cE}$ and
    $\|\bfn_{\bar\cE} - \bfn\|_{C^1} =
    O(\eps^2)$.}
  Certainly, integrability of the caustics $\Gm_{1/q}$
  (where $q = \lceil k/2\rceil$ and $\lceil\cdot\rceil$
  denotes the ceiling function) is preserved {by the
    perturbation $\bfn_{\bar\cE}$}.  Therefore, by
  Lemma~\ref{l_projection-estimate}, if $4 < k \le \eps^{-1/9}$
  { we gather that  $|\int\bfn_{\bar\cE}\mu e_k| \le C
    k^8\|\bfn_{\bar\cE}\|_{C^\rcloseness}^2$,} which gives:
  \begin{align} \label{alm-orthogonal} \left|\eps\int
      e_j(x)\, \, \mu(x)\, e_k(x)\,dx\right| \le Ck^8
    \eps^2\le C\eps^{10/9}.
  \end{align}
  Since $\eps$ can be chosen arbitrarily and the functions
  $\{e_k\}$ do not depend on the perturbation, but only on
  $\cE_e$, our lemma follows.
\end{proof}
\begin{remark*}
  Lemma~\ref{l_orthogonality} can be seen as an
  orthogonality relation with respect to the $L^2$ inner
  product with weight $\mu$.
\end{remark*}

\section{The Deformed Fourier basis}
\label{main-nonfamily}
In the previous section we completed the definition of the
Deformed Fourier modes by introducing the first $5$ modes;
let $\basis: = (e_0,e_1,\cdots,e_j,\cdots)$.  Let us also
introduce the corresponding Fourier Modes $\eF_j$ so that
$\eF_0 = 1$ and, for $j > 0$,
$\eF_{2j} = \sqrt2\,\cos(2\pi j\cdot)$ and
$\eF_{2j-1} = \sqrt2\,\sin(2\pi j\cdot)$.  Observe that we
choose the normalization in such a way that $(\eF_j)$ is an
orthonormal basis.

Let us define the following operator acting on $L^2$:
\begin{align}\label{e_definitionL}
{\cL: v\mapsto \sum_{j = 0}^{\infty}\left[\int
    \eF_jvdx\right]e_j = \sum_{j = 0}^{\infty} \hat v_j e_j}
  \end{align}
where $\hat v_j$ is the $j$-th Fourier coefficient of $v$,
i.e. $v = \sum_{j = 0}^{\infty}\hat v_j\eF_j$.  In the
sequel we will denote by $\|\cdot\|\nol2$ the usual
operator norm in $L^2$ given by:
\begin{align*}
  \|T\|\nol2\ = \ \sup_{f:\,\|f\|\nl2 \le1}\ \|T\, f\|\nl2.
\end{align*}%
\begin{proposition}\label{mixed-L2-basis}
  Assume that $e_* > 0$ is so small that
  \begin{align}\label{eccentr-smallness-assumption}
    C^*(e_*)\sqrt{1+\frac{\pi^2}3} &< 1, &\text{where $C^*(e)$
    is defined in Lemma~\ref{error-estimate-2}}.
  \end{align}
  Then, if $\cE_e$ is an ellipse of eccentricity
  $0\le e \le e_*$ and perimeter $1$, the operator $\cL$ is
  bounded and invertible 
  as an operator from $L^2$ to $L^2$.  {In particular,
    $\basis$ is a basis of $L^2$.}
\end{proposition}
\begin{proof}
  First of all, observe that if
    $\|\cL-\Id\|\nol2 < 1$, then $\cL$ is an bounded
    invertible operator with a bounded inverse.  Notice that
    for any $v\in L^2$,
    $v = \sum_{j = 0}^{\infty}\hat v_j\eF_j$:
    \begin{align*}
      [\cL-\Id](v) = \sum_{j = 0}^\infty \hat v_j(e_j-\eF_j).
    \end{align*}
    By definition, then:
    \begin{align*}
      \|\cL-\Id\|\nol2 &= \sup_{v: \|v\|\nl2 \le 1}
                         \|[\cL-\Id]v\|\nl2,%
    \end{align*}
    hence, by the Cauchy Inequality
    \begin{align*}
      \|[\cL-\Id]v\|\nl2 &\le \sum_{j =
                           2N+1}^\infty|\hat v_j|\|e_j-\eF_j\|\nl2\\
                         &\le \left[\sum_{j =
                           0}^\infty |\hat v_j|^2\right]^{1/2}\left[\sum_{j = 0}^\infty\|e_j-\eF_j\|\nl2^2\right]^{1/2}.
    \end{align*}
    Thus, using Parseval's identity we conclude that
    $\sum_{j = 0}^\infty|\hat v_j|^2 = \|v\|\nl2^2\leq1$.
    Therefore, by Lemma~\ref{error-estimate-2}, the
    definition of $e_j$ and $\eF_j$ and
    using~\eqref{eccentr-smallness-assumption} we finally
    conclude that:
  \[
  \|\cL-\Id\|\nol2 \le C^{*}(e)\left[1+2\sum_{j=1}^\infty
    \frac1{j^2}\right]^{1/2} < 1.\qedhere
  \]
\end{proof}
{ Let us now define, for any
  $q\ge0$ \begin{align}\label{e_defTilden} \tilde
    n_q&:=\int\bfn(x)\mu(x)e_q(x)dx.  \end{align} Notice
  that these numbers are \emph{not} the coefficients of the
  decomposition of $\bfn\cdot\mu$ in the basis $\basis$,
  because $\basis$ is not an orthonormal basis.  Despite
  this limitation, it is possible to obtain the following
  useful bound.}  {\begin{corollary}\label{c_parseval} The
    following estimate holds \begin{align*}
      \|\bfn\|_{L^2}^2\le C\sum_{q = 0}^{\infty}|\tilde
      n_q|^2.  \end{align*} \end{corollary} \begin{proof}
    Let us define the operator $\cL_\mu$ from $L^2\to L^2$
    given by \begin{align*} \cL_\mu v(x) = \mu(x)\cdot[\cL
      v](x), \end{align*} where $\cL$ is defined
    in~\eqref{e_definitionL}.  Then by
    Proposition~\ref{mixed-L2-basis} and since both $\mu(x)$
    and $\mu(x)^{-1}$ are bounded and analytic, we conclude
    that $\cL_\mu:L^2\to L^2$ is a bounded invertible
    operator; therefore, so is its adjoint $\cL_\mu^*$.
    Hence, using Parseval's Identity:
    \begin{align*}
    \|\bfn\|^2_{L^2} &=
                       \|(\cL_\mu^*)^{-1}\cL_\mu^*\bfn\|^2_{L^2} \le
                       C\|\cL_\mu^*\bfn\|^2_{L^2} = C
                       \sum_{q =
                       0}^{\infty}\left|\int\cL_\mu^*(\bfn)\eF_q\right|^2\\
                     &= C
                       \sum_{q =
                       0}^{\infty}\left|\int\bfn\cL_\mu(\eF_q)\right|^2\le C\sum_{q = 0}^{\infty}\left|\int\bfn
                       \mu e_q\right|^2
  \end{align*}
  where we used the fact that
  $\cL_\mu\eF_q = \mu\cdot \cL\eF_q = \mu\cdot e_q$.
\end{proof}}
\section{Proof of the Main Theorem}\label{s_theProof}
The proof of our Main Theorem relies on the following
approximation result.
\begin{lemma}\label{l_inductive}
  Let $e_*$ be sufficiently small, so
  that~\eqref{eccentr-smallness-assumption} holds and let
  $\cE_e$ be an ellipse of perimeter $1$ and eccentricity
  $e\in[0,e_*]$.  Let $\Om$ be a rationally integrable
  $C^\rsmoothness$ deformation of $\cE_e$ identified by a
  $C^\rsmoothness$ function $\bfn(x)$, i.e.
  $\partial \Om:=\cE_e+ \bfn.$ Then there exists an ellipse
  $\bar\cE$ and $\bar\bfn$ so that
  $\partial\Om = \bar\cE+\bar\bfn$ and
  \begin{align*}
    \|\bar\bfn \|_{C^\rcloseness}\le
    C(e,\|\bfn\|_{C^\rsmoothness}) \, \|\bfn\|^{\rsuperexp}_{C^\rcloseness}. \ \
  \end{align*}
\end{lemma}
Before giving the proof of Lemma~\ref{l_inductive}, let us
use it to prove our Main Theorem, which we now
state in a (slightly) stronger version:
\begin{theorem}\label{t_main}
  Let $e_*$ be sufficiently small, so
  that~\eqref{eccentr-smallness-assumption}.  For any
  $0 < e_0 < e_*$ and $K > 0$, there exists $\eps > 0$ so
  that, for any $0\le e \le e_0$, any rationally integrable
  $C^\rsmoothness$-smooth domain $\Om$ so that $\partial\Om$
  is $C^{\rsmoothness}$-$K$-close and
  $C^\rcloseness$-$\eps$-close to $\cE_e$ is an ellipse.
\end{theorem}
\begin{proof} %
  To ease our notations, let us drop the subscript $e$ and
  let $\cE = \cE_e$.  Let us fix $K > 0$ arbitrarily and
  $\eps > 0$ sufficiently small to be specified later.  Denote
  with $\kEll_{\eps}(\cE)$ the set of ellipses (not
  necessarily of perimeter $1$) whose $C^0$-Hausdorff
  distance from $\cE$ is not larger than $2\eps$, i.e.
  \begin{align*}
  \kEll_{\eps}(\cE) = \{\cE'\subset\R^2, \hDist(\cE,\cE')\le2\eps\}.
  \end{align*}
  We assume $\eps$ so small (depending on $e_0$) that any
  $\cE'\in \kEll_\eps(\cE)$ has length
  $\ell_{\cE'}\in[3/4,5/4]$ and eccentricity $e'\in[0,e_*]$.
  Recall that any ellipse in $\R^2$ can be parametrized by
  $5$ real quantities (e.g. the coefficients of the
  corresponding quadratic equation): let $\kPar_\eps(\cE)$
  be the set of parameters $a\in\R^5$ corresponding to
  ellipses in $\kEll_\eps(\cE)$; then $\kPar_\eps(\cE)$ is
  compact.

  Let now $\bfn$ be a $C^\rsmoothness$ perturbation with
  $\|\bfn\|_{C^{\rsmoothness}} < K$ and
  $\|\bfn\|_{C^\rcloseness} < \eps$ and consider the domain
  $\Om$ given by
  \begin{align*}
    \partial\Om = \cE+\bfn.
  \end{align*}
  For any $5$-tuple of parameters $a\in \kPar$ we associate
  the corresponding ellipse $\cE_a$ and perturbation
  $\bfn_a$ so that $\partial\Om = \cE_a+\bfn_a$.  Observe
  that the Lazutkin tubular coordinates $(x,n)$ of $\Om$
  change analytically with respect to $a$; we conclude that
  $\bfn_a$ also varies analytically with respect to $a$.  In
  particular, we can assume $\eps$ so small that for any
  $a\in\kPar_\eps(\cE)$,
  $\|\bfn_a\|_{C^{\rsmoothness}} < 2K$.  Moreover, the
  function $a\mapsto\|\bfn_a\|_{C^{\rcloseness}}$ is a
  continuous function and as such it will have a minimum,
  which we denote by $a_*\in \kPar_\eps(\cE)$.  To ease our
  notation, let $\cE_* = \cE_{a_*}$ and correspondingly
  $\bfn_* = \bfn_{a_*}$; then by definition:
  \begin{align*}
    0 \le \|\bfn_*\|_{C^{\rcloseness}} \le
    \|\bfn\|_{C^{\rcloseness}} \le \eps.
  \end{align*}
  Modulo a possible linear rescaling (which also rescales
  linearly $\bfn$, since the Lazutkin perimeter is
  normalized to be $1$) we can assume that $\cE_*$ has
  perimeter $1$; we thus, apply Lemma~\ref{l_inductive} to
  $\cE_*$ and $\bfn_*$ obtaining $\bar\cE_*$ and
  $\bar\bfn_*$.  But if $\eps$ is small enough, then there
  exists $\varrho\in(0,1)$ so that
  $\|\bar\bfn_*\|_{C^{\rcloseness}} \le \varrho
  \|\bfn_*\|_{C^{\rcloseness}}$.
  Hence, by the triangle inequality,
  \begin{align*}
    \hDist(\cE,\bar\cE_*)\le\hDist(\cE,\Om)+\hDist(\Om,\bar\cE_*)
    \le (1+\varrho)\eps < 2\eps
  \end{align*}
  thus $\bar\cE_*\in\kEll_\eps(\cE)$.  Since
  $\|\bfn_*\|_{C^{\rcloseness}}$ was minimal, we conclude
  that
  $\|\bfn_*\|_{C^{\rcloseness}}
  =\|\bar\bfn_*\|_{C^{\rcloseness}} = 0$,
  i.e.  $\Om = \cE_*$ is an ellipse.
\end{proof}
We conclude this article by giving the
\begin{proof}[Proof of Lemma~\ref{l_inductive}]
  Observe that Lemma~\ref{l_orthogonality} implies that the
  vectors $\{e_j\,:\,0\le j\le 4\}$ are $\mu$-orthogonal to
  the subspace generated by $\{e_j\,:\,j>4\}$.

  Now, let us decompose
  \begin{align}\label{e_decompositionBfn}
    \bfn(x) = \bfn^{(5)}(x) +\bfn^\perp(x)
  \end{align}
  where $\bfn^\perp$ is $\mu$-orthogonal to the subspace
  spanned by $\{e_j\,:\,0\le j\le 4\}$ and $\bfn^{(5)}$ is
  its complement; then $\bfn^{(5)} = \sum_{j = 0}^4a_je_j$
  for some $(a_j)_{0\le j\le 4}$.

  We claim that $|a_j| < C\|\bfn\|_{C^\rcloseness}$, where
  $C = C(e)$ depends on the eccentricity $e$ only.  By
  $\mu$-orthogonality we have
  \begin{align*}
    \|\bfn^{(5)}\|\nlmu2^2+\|\bfn^{\perp}\|\nlmu2^2 =
    \|\bfn\|\nlmu2^2\le  C\|\bfn\|_{C^\rcloseness}^2,
  \end{align*}
  where $C = C(e)$ and $\|\cdot\|\nlmu2$ denotes the $L^2$
  norm induced by the inner product with weight $\mu$, i.e.
  $\|f\|\nlmu2 = \|\sqrt\mu f\|\nl2$; this norm is clearly
  equivalent to the standard $L^2$ norm.  In particular, we
  have $\|\bfn^{(5)}\|\nl2\le C\|\bfn\|_{C^\rcloseness}$,
  which implies our claim.

  Since $e_j$ is analytic for $0\le j\le 4$, we also have
  \begin{align}\label{e_smallnessEllipse}
    \|\bfn^{(5)}\|_{C^\rsmoothness} < C\|\bfn\|_{C^\rcloseness}.
  \end{align}
  We now claim that
  \begin{align}\label{e_claimSmallnPerp}
    \|\bfn^\perp\|_{C^\rcloseness}\le C(e,\|\bfn\|_{C^\rsmoothness})\|{\bfn}\|_{C^\rcloseness}^{\rsuperexp}
  \end{align}
  where $C$ above depends monotonically on
  $\|\bfn\|_{C^{\rsmoothness}}$.  The above estimate allows
  to conclude the proof of our result as we now describe.

  Let $\bar\cE$ be the ellipse obtained by applying
  Corollary~\ref{c_combiningLemmata} to $\cE$ and
    $\bfn^{(5)}$; recall that by construction
    $\bar\cE = \cE + \bfn_{\bar\cE}$ and,
    using~\eqref{e_smallnessEllipse}, we obtain the bound
  \begin{align}\label{e_estimateRigid}
    \|\bfn_{\bar\cE}-\bfn^{(5)}\|_{C^{\rcloseness}}\le C\|\bfn\|_{C^\rcloseness}^2.
  \end{align}
  Then let $\Om = \bar\cE+\bar\bfn$; by
  Lemma~\ref{l_approximationLemma} we conclude that for some
  $C$ depending on $e$ only,
  \begin{align*}
    \|\bar\bfn\|_{C^\rcloseness}\le C\|\bfn-\bfn_{\bar\cE}\|_{C^{\rcloseness}}
    = C\|\bfn^{(5)} -\bfn_{\bar\cE}+\bfn^{\perp}\|_{C^{\rcloseness}}.
  \end{align*}
  By the triangle inequality,
  using~\eqref{e_claimSmallnPerp}
  and~\eqref{e_estimateRigid} we gather that
  \begin{align*}
    \|\bfn^{(5)} -\bfn_{\bar\cE}+\bfn^{\perp}\|_{C^\rcloseness} <
    C(e,\|\bfn\|_{C^\rsmoothness})\|{\bfn}\|_{C^\rcloseness}^{\rsuperexp},
  \end{align*}
  which completes the proof of our lemma.

  We are left with the proof of~\eqref{e_claimSmallnPerp}:
  we first show that the component $\bfn^\perp$ of the
  decomposition~\eqref{e_decompositionBfn} is $L^2$-small
  and, later, we will deduce that it is indeed
  $C^\rcloseness$-small.  Applying
  Corollary~\ref{c_parseval} to $\bfn^\perp$ and taking into
  account its orthogonality to the first $5$ modes (see
  Lemma~\ref{l_orthogonality}) we obtain
   \begin{align*}
     \|\bfn^\perp\|^2_{L^2} &\le C\sum_{q = 5}^{\infty}|\tilde n_q|^2,
   \end{align*}
   where $\tilde n_q$ has been defined in~\eqref{e_defTilden}.

  Fix $\alpha < 1/8$ to be specified later and let
  $q_0=[ \|{\bf n}\|_{C^\rcloseness}^{-\alpha}]$, where
  $[x]$ denotes the integer part of $x$; by
  Lemma~\ref{l_projection-estimate}, for any $4<q\le q_0$,
  we have
  \begin{align*}
    |\tilde n_q|\le Cq^8 \|\bfn\|_{C^\rcloseness}^2\le C\|{\bf n}\|_{C^\rcloseness}^{2-8\alpha},
  \end{align*}
  where $C$ depends on $e$ and on $\|\bfn\|_{C^5}$ only.
  Then, summing over $5\le q\le q_0$, we obtain
  \begin{align*}
    \sum_{q = 5}^{q_0} |\tilde n_q|^2\le
    C\|\bfn\|_{C^\rcloseness}^{4-17\alpha}.
  \end{align*}
  On the other hand, Lemma~\ref{coeff-decay-simpler} gives:
  \begin{align*}
    |\tilde n_q|^2\le C  \frac{\|\bfn\|_{C^\rcloseness}^2}{q^2};
  \end{align*}
  therefore, summing over $q > q_0$ we conclude that
  \begin{align*}
    \sum_{q = q_0+1}^\infty|\tilde n_q|^2 \le C\|\bfn\|_{C^\rcloseness}^{2+\alpha}.
  \end{align*}
  Combining the two above estimates and optimizing for
  $\alpha$ (i.e.\ choosing $\alpha = 1/9$), we conclude that
  $\|\bfn^\perp\|\nl2\le C\|\bfn\|_{C^\rcloseness}^{19/18}$.

  In order to upgrade this $L^2$ estimate to a
  $C^\rcloseness$ estimate, first, observe that we have:
  \begin{align*}
    \|\bfn^\perp\|_{C^\rcloseness}\le %
    \|D\bfn^\perp\|_{L^1}+\|D^2\bfn^\perp\|_{L^1}\le
    \|D\bfn^\perp\|_{L^2}+\|D^2\bfn^\perp\|_{L^2}.
  \end{align*}
  We then use standard Sobolev interpolation inequalities
  (see e.g.~\cite{GT}): for any $\dt>0$ and any
  $1\le j\le 2$ we have,
  \begin{align*}
    \|D^j \bfn^\perp\|_{L^2}\le%
    C \left[\dt \|\bfn^\perp\|_{C^\rsmoothness}+
    \dt^{-j/(\rsmoothness-j)} \|\bfn^\perp\|_{L^2}\right].
  \end{align*}
  Optimizing the above estimate\footnote{ The number
    $\rsmoothness$ has indeed been chosen to be minimal
    among those for which the above interpolation inequality
    provides an useful bound.\label{fn_rsmoothness}}, we choose
  $\dt=\|\bfn\|_{C^\rcloseness}^{\rsuperexp}$.
  Observe that $\|\bfn^\perp\|_{C^{\rsmoothness}}$ is
  uniformly bounded using~\eqref{e_smallnessEllipse};
  we thus conclude that~\eqref{e_claimSmallnPerp} holds.
\end{proof}
\bibliographystyle{abbrv} \bibliography{birkhoff}

\begin{thebibliography}{10}

\bibitem{AM}
K.~G. Andersson and R.~B. Melrose.
\newblock The propagation of singularities along gliding rays.
\newblock {\em Invent. Math.}, 41(3):197--232, 1977.

\bibitem{Bi}
M.~Bialy.
\newblock Convex billiards and a theorem by {E}. {H}opf.
\newblock {\em Math. Z.}, 214(1):147--154, 1993.

\bibitem{Bf}
G.~D. Birkhoff.
\newblock {\em Dynamical systems}.
\newblock With an addendum by Jurgen Moser. American Mathematical Society
  Colloquium Publications, Vol. IX. American Mathematical Society, Providence,
  R.I., 1966.

\bibitem{Bu}
L.~A. Bunimovich.
\newblock On absolutely focusing mirrors.
\newblock In {\em Ergodic theory and related topics, {III} ({G}\"ustrow,
  1990)}, volume 1514 of {\em Lecture Notes in Math.}, pages 62--82. Springer,
  Berlin, 1992.

\bibitem{Chaza}
J.~Chazarain.
\newblock Formule de {P}oisson pour les vari\'et\'es riemanniennes.
\newblock {\em Invent. Math.}, 24:65--82, 1974.

\bibitem{Dui}
J.~J. Duistermaat and V.~W. Guillemin.
\newblock The spectrum of positive elliptic operators and periodic
  bicharacteristics.
\newblock {\em Invent. Math.}, 29(1):39--79, 1975.

\bibitem{GT}
D.~Gilbarg and N.~S. Trudinger.
\newblock {\em Elliptic partial differential equations of second order}.
\newblock Classics in Mathematics. Springer-Verlag, Berlin, 2001.
\newblock Reprint of the 1998 edition.

\bibitem{GWW}
C.~Gordon, D.~L. Webb, and S.~Wolpert.
\newblock One cannot hear the shape of a drum.
\newblock {\em Bull. Amer. Math. Soc. (N.S.)}, 27(1):134--138, 1992.

\bibitem{GM1}
V.~Guillemin and R.~Melrose.
\newblock An inverse spectral result for elliptical regions in {${\bf
  R}\sp{2}$}.
\newblock {\em Adv. in Math.}, 32(2):128--148, 1979.

\bibitem{Gu}
E.~Gutkin.
\newblock Billiard dynamics: an updated survey with the emphasis on open
  problems.
\newblock {\em Chaos}, 22(2):026116, 13, 2012.

\bibitem{HZ}
H.~Hezari and S.~Zelditch.
\newblock {$C\sp \infty$} spectral rigidity of the ellipse.
\newblock {\em Anal. PDE}, 5(5):1105--1132, 2012.

\bibitem{Ka}
M.~Kac.
\newblock Can one hear the shape of a drum?
\newblock {\em Amer. Math. Monthly}, 73(4, part II):1--23, 1966.

\bibitem{L}
V.~F. Lazutkin.
\newblock Existence of caustics for the billiard problem in a convex domain.
\newblock {\em Izv. Akad. Nauk SSSR Ser. Mat.}, 37:186--216, 1973.

\bibitem{PR}
S.~Pinto-de Carvalho and R.~Ram{\'{\i}}rez-Ros.
\newblock Non-persistence of resonant caustics in perturbed elliptic billiards.
\newblock {\em Ergodic Theory Dynam. Systems}, 33(6):1876--1890, 2013.

\bibitem{Popov}
G.~Popov.
\newblock Invariants of the length spectrum and spectral invariants of planar
  convex domains.
\newblock {\em Comm. Math. Phys.}, 161(2):335--364, 1994.

\bibitem{Po}
H.~Poritsky.
\newblock The billiard ball problem on a table with a convex boundary---an
  illustrative dynamical problem.
\newblock {\em Ann. of Math. (2)}, 51:446--470, 1950.

\bibitem{RR}
R.~Ram{\'{\i}}rez-Ros.
\newblock Break-up of resonant invariant curves in billiards and dual billiards
  associated to perturbed circular tables.
\newblock {\em Phys. D}, 214(1):78--87, 2006.

\bibitem{Sa}
P.~Sarnak.
\newblock Determinants of {L}aplacians; heights and finiteness.
\newblock In {\em Analysis, et cetera}, pages 601--622. Academic Press, Boston,
  MA, 1990.

\bibitem{Sunada}
T.~Sunada.
\newblock Riemannian coverings and isospectral manifolds.
\newblock {\em Ann. of Math. (2)}, 121(1):169--186, 1985.

\bibitem{Taba}
S.~Tabachnikov.
\newblock {\em Geometry and billiards}, volume~30 of {\em Student Mathematical
  Library}.
\newblock American Mathematical Society, Providence, RI; Mathematics Advanced
  Study Semesters, University Park, PA, 2005.

\bibitem{Ta}
M.~B. Tabanov.
\newblock New ellipsoidal confocal coordinates and geodesics on an ellipsoid.
\newblock {\em J. Math. Sci.}, 82(6):3851--3858, 1996.
\newblock Algebra, 3.

\bibitem{Tre}
D.~Treschev.
\newblock Billiard map and rigid rotation.
\newblock {\em Phys. D}, 255:31--34, 2013.

\bibitem{Vi}
M.-F. Vign{\'e}ras.
\newblock Vari\'et\'es riemanniennes isospectrales et non isom\'etriques.
\newblock {\em Ann. of Math. (2)}, 112(1):21--32, 1980.

\bibitem{Woj}
M.~P. Wojtkowski.
\newblock Two applications of {J}acobi fields to the billiard ball problem.
\newblock {\em J. Differential Geom.}, 40(1):155--164, 1994.

\end{thebibliography}
\end{document}